\documentclass[a4paper,11pt,reqno]{amsart}
\usepackage{enumerate}
\usepackage{amssymb, amsmath}
\usepackage{mathrsfs}
\usepackage{amscd}
\usepackage[active]{srcltx}
\usepackage{verbatim}
\usepackage[colorlinks,linkcolor={black},citecolor={blue},urlcolor={black}]{hyperref}

\theoremstyle{plain}
\newtheorem{theorem}{Theorem}[section]
\newtheorem{corollary}[theorem]{Corollary}
\newtheorem{lemma}[theorem]{Lemma}
\newtheorem{proposition}[theorem]{Proposition}

\newtheorem{definition}[theorem]{Definition}
\newtheorem{assumption}[theorem]{Assumption}

\newtheorem*{definition*}{Definition}

\theoremstyle{remark}
\newtheorem{remark}[theorem]{Remark}
\newtheorem{example}[theorem]{Example}

\newtheorem*{claim*}{Claim}
\newtheorem*{remark*}{Remark}
\newtheorem*{example*}{Example}
\newtheorem*{notation*}{Notation}

\numberwithin{equation}{section}

\def\N{{\mathbb N}}
\def\Z{{\mathbb Z}}
\def\R{{\mathbb R}}

\newcommand{\E}{{\mathbb E}}

\newcommand{\one}{{{\bf 1}}}

\newcommand{\eps}{\varepsilon}
\renewcommand{\phi}{\varphi}

\newcommand{\hDelta}{\hat\Delta}
\newcommand{\dd}{\; \mathrm{d}}
\newcommand{\xx}{\mathbf{x}}
\newcommand{\yy}{\mathbf{y}}
\newcommand{\cc}{\mathbf{c}}

\DeclareMathOperator{\Ran}{Ran}

\DeclareMathOperator{\Ker}{Ker}

\newcommand{\ip}[1]{\langle {#1}\rangle}
\newcommand{\bip}[1]{\big\langle {#1}\big\rangle}
\newcommand{\norm}[1]{\parallel {#1}\parallel}
\newcommand{\abs}[1]{\vert {#1}\vert}

\DeclareMathOperator{\grad}{grad}

\DeclareMathOperator{\Ric}{Ric}
\DeclareMathOperator{\Lip}{Lip}
\DeclareMathOperator{\Hess}{Hess}
\DeclareMathOperator{\Leb}{Leb}

\DeclareMathOperator{\HcWI}{H{\cW}I}
\DeclareMathOperator{\mLSI}{MLSI}
\DeclareMathOperator{\mTal}{T_\cW}
\DeclareMathOperator{\Poinc}{P}

\newcommand{\ddt}{\frac{\mathrm{d}}{\mathrm{d}t}}
\newcommand{\ddtr}{\frac{\mathrm{d}^+}{\mathrm{d}t}}
\newcommand{\ddtrz}{\left.\frac{\mathrm{d}^+}{\mathrm{d}t}\right\vert_{t=0}}
\newcommand{\ddhr}{\frac{\mathrm{d}^+}{\mathrm{d}h}}
\newcommand{\ddtt}{\frac{\mathrm{d^2}}{\mathrm{d}t^2}}
\newcommand{\ddr}{\frac{\mathrm{d}}{\mathrm{d}r}}

\newcommand{\cQ}{\mathcal{Q}}

\newcommand{\cM}{\mathcal{M}}
\newcommand{\cH}{\mathcal{H}}
\newcommand{\cB}{\mathcal{B}}

\newcommand{\cW}{\mathcal{W}}
\newcommand{\cI}{\mathcal{I}}

\newcommand{\Proj}{\mathcal{P}}

\newcommand{\cG}{\mathcal{G}}
\newcommand{\cK}{\mathcal{K}}
\newcommand{\cX}{\mathcal{X}}
\newcommand{\cE}{\mathcal{E}}
\newcommand{\cA}{\mathcal{A}}

\newcommand{\CE}{\mathcal{CE}}

\newcommand{\cF}{\mathcal{F}}

\newcommand{\cP}{\mathscr{P}}
\newcommand{\PX}{\cP(\cX)}
\newcommand{\PXs}{\cP_*(\cX)}
\newcommand{\hrho}{\hat\rho}
\newcommand{\ulrho}{\breve{\rho}}
\newcommand{\ulV}{\breve{V}}

\renewcommand{\tilde}{\widetilde}

\newcommand{\nl}{{non-local }}

\begin{document}

\title
[ Ricci curvature of finite Markov chains]
{Ricci curvature of finite Markov chains via convexity of the entropy}

\author{Matthias Erbar}
\author{Jan Maas}
\address{
University of Bonn\\
Institute for Applied Mathematics\\
Endenicher Allee 60\\
53115 Bonn\\
Germany}
\email{erbar@iam.uni-bonn.de}
\email{maas@iam.uni-bonn.de}

\thanks{JM is supported by Rubicon subsidy 680-50-0901 of the
  Netherlands Organisation for Scientific Research (NWO)}


 \begin{abstract}
   We study a new notion of Ricci curvature that applies to
   Markov chains on discrete spaces. This notion relies on
   geodesic convexity of the entropy and is analogous to the one
   introduced by Lott, Sturm, and Villani for geodesic measure spaces.
   In order to apply to the discrete setting, the role of the
   Wasserstein metric is taken over by a different metric, having the
   property that continuous time Markov chains are gradient flows of
   the entropy.

   Using this notion of Ricci curvature we prove 
   discrete analogues of fundamental results by Bakry--\'Emery and
   Otto--Villani.
   Furthermore we show that Ricci curvature bounds are preserved under
   tensorisation. As a special case we obtain the sharp Ricci
   curvature lower bound for the discrete hypercube.
 \end{abstract}


\date\today

\maketitle
 
\tableofcontents

\section{Introduction}
\label{sec:intro}

In two independent contributions Sturm \cite{S06} and Lott and Villani
\cite{LV09} solved the long-standing open problem of defining a
synthetic notion of Ricci curvature for a large class of metric
measure spaces.

The key observation, proved in \cite{vRS05}, is that on a Riemannian
manifold $\cM$, the Ricci curvature is bounded from below by some
constant $\kappa \in \R$, if and only if the Boltzmann-Shannon entropy
$\cH(\rho) = \int\rho \log \rho \dd \mathrm{vol}$ is $\kappa$-convex
along geodesics in the $L^2$-Wasserstein space of probability measures
on $\cM$.
The latter condition does not appeal to the geometric structure $\cM$,
but only requires a metric (to define the $L^2$-Wasserstein metric
$W_2$) and a reference measure (to define the entropy
$\cH$). Therefore this condition can be used in order to define a
notion of Ricci curvature lower boundedness on more general metric
measure spaces. This notion turns out to be stable under
Gromov-Hausdorff convergence and it implies a large number of
functional inequalities with sharp constants. The theory of metric
measure spaces with Ricci curvature bounds in the sense of Lott,
Sturm, and Villani is still under active development
\cite{AGS11a,AGS11b}.

However, the condition of Lott-Sturm-Villani does not apply if the
$L^2$-Wasserstein space over $\cX$ does not contain
geodesics. Unfortunately, this is the case if the underlying space is
discrete (even if the underlying space consists of only two
points). The aim of the present paper is to develop a variant of the
theory of Lott-Sturm-Villani, which does apply to discrete spaces.

In order to circumvent the nonexistence of Wasserstein geodesics, we
replace the $L^2$-Wasserstein metric by a different metric $\cW$,
which has been introduced in \cite{Ma11}. There it has been shown that
the heat flow associated with a Markov kernel on a finite set is the gradient flow of the entropy with respect to $\cW$ (see also the independent work \cite{CHLZ11}
containing related results for Fokker-Planck equations on graphs, as well as \cite{Mie11a}, where this gradient flow structure has been discovered in the setting of reaction-diffusion systems). In
this sense, $\cW$ takes over the role of the Wasserstein metric, since
it is known since the seminal work by Jordan, Kinderlehrer, and Otto
that the heat flow on $\R^n$ is the gradient flow of the entropy
\cite{JKO98} (see \cite{AGS11a,Erb10,FSS10,GKO10,OhSt09,O01}
for variations and generalisations). Convexity along $\cW$-geodesics may thus be regarded as a discrete analogue of McCann's displacement convexity \cite{McC97}, which corresponds to convexity along $W_2$-geodesics in a continuous setting.

Since every pair of probability densities on $\cX$ can be joined by a
$\cW$-geodesic, it is possible to define a notion of Ricci curvature
in the spirit of Lott-Sturm-Villani by requiring geodesic convexity of
the entropy with respect to the metric $\cW$. This possibility has
already been indicated in \cite{Ma11}. We shall show that this notion
of Ricci curvature shares a number of properties which make the LSV
definition so powerful: in particular, it is stable under
tensorisation and implies a number of functional inequalities,
including a modified logarithmic Sobolev inequality, and a
Talagrand-type inequality involving the metric $\cW$.

\subsection*{Main results}
Let us now discuss the contents of this paper in more detail. We work
with an irreducible Markov kernel $K : \cX \times \cX \to \R_+$ on a
finite set $\cX$, i.e., we assume that 
\begin{align*}
\sum_{y\in \cX} K(x,y) = 1
\end{align*}
for all $x\in \cX$, and that for every $x,y \in \cX$ there exists a
sequence $\{x_i\}_{i=0}^n \in \cX$ such that $x_0 = x$, $x_n = y$ and
$K(x_{i-1}, x_i) > 0$ for all $1 \leq i \leq n$. Basic Markov chain theory
guarantees the existence of a unique stationary probability measure
(also called steady state) $\pi$ on $\cX$, i.e.,
\begin{align*}
\sum_{x \in \cX}\pi(x) =1\qquad  \text{and} \qquad\pi(y) = \sum_{x \in \cX} \pi(x)K(x,y)
\end{align*}
for all $y \in \cX$. We assume that $\pi$ is \emph{reversible} for
$K$, which means that the detailed balance equations
\begin{align}
\label{eq:detailed-balance}
 K(x,y) \pi(x) =  K(y,x) \pi(y)
\end{align}
hold for $x, y \in \cX$.

Let 
\begin{align*}
 \PX := \Big\{ \, \rho : \cX \to \R_+ \ | \ 
     \sum_{x \in \cX} \pi(x) \rho(x)  = 1 \, \Big\}
\end{align*}
be the set of \emph{probability densities} on $\cX$. The subset
consisting of those probability densities that are strictly positive
is denoted by $\PXs$.  We consider the metric $\cW$ defined for
$\rho_0, \rho_1 \in \PX$ by
\begin{align*}
 \cW(\rho_0, \rho_1)^2
   := \inf_{\rho, \psi} 
   \bigg\{  \frac12   \int_0^1 
  \sum_{x,y\in \cX} (\psi_t(x) - \psi_t(y))^2
    		 \hat\rho_t(x,y)  K(x,y)\pi(x)
      \dd t 
          \bigg\}\;,
\end{align*}
where the infimum runs over all sufficiently regular curves $\rho :
[0,1] \to \PX$ and $\psi : [0,1] \to \R^\cX$ satisfying the
`continuity equation'
\begin{align} \label{eq:cont}
 \begin{cases}
 \displaystyle\ddt \rho_t(x) 
   + \displaystyle\sum_{y \in \cX} ( \psi_t(y) - \psi_t(x) ) \hat\rho_t(x,y) K(x,y) ~=~0\qquad \forall x \in \cX\;, \\ 
  \rho(0) = \rho_0\;, \qquad \rho(1)  = \rho_1\;.
 \end{cases}
\end{align}
Here, given $\rho\in\PX$, we write
$\hat\rho(x,y):=\int_0^1\rho(x)^{1-p}\rho(y)^p\dd p$ for the
logarithmic mean of $\rho(x)$ and $\rho(y)$. The relevance of the logarithm mean in this setting is due to the identity
\begin{align*}
 \rho(x) - \rho(y) = \hat\rho(x,y) ( \log \rho(x) - \log \rho(y) )\;,
\end{align*}
which somehow compensates for the lack of a `discrete chain rule'.
The definition of $\cW$ can be regarded as a discrete analogue of the Benamou-Brenier formula \cite{BB00}.
Let us remark that if $t \mapsto \rho_t$ is differentiable at some $t$ and $\rho_t$ belongs to $\cP_*(\cX)$, then the continuity equation \eqref{eq:cont} is satisfied for some $\psi_t \in \R^\cX$, which is unique up to an additive constant (see \cite[Proposition 3.26]{Ma11}).

Since the metric $\cW$ is Riemannian in the interior $\PXs$, it makes
sense to consider gradient flows in $(\PXs, \cW)$ and it has been
proved in \cite{Ma11} that the heat flow associated with the
continuous time Markov semigroup $P_t = e^{t(K-I)}$ is the gradient
flow of the entropy
\begin{align} \label{eq:entropy}
 \cH(\rho) = \sum_{x \in \cX} \pi(x) \rho(x) \log \rho(x)\;,
\end{align}
with respect to the Riemannian structure determined by $\cW$. 
 
In this paper we shall show that every pair of densities $\rho_0, \rho_1 \in \PX$ can be joined by a constant speed geodesic. Therefore the following definition in the spirit of Lott-Sturm-Villani seems natural.
\begin{definition}\label{def:intro-Ricci}
  We say that $K$ has \emph{\nl Ricci curvature bounded from below by $\kappa \in \R$} 
  if for any constant speed geodesic $\{\rho_t\}_{t \in [0,1]}$ in $(\PX, \cW)$
  we have
\begin{align*}
  \cH(\rho_t) \leq
  (1-t) \cH(\rho_0) + t \cH(\rho_1) - \frac\kappa{2} t(1-t) \cW(\rho_0, \rho_1)^2\;.
\end{align*}
In this case, we shall use the notation
\begin{align*}
 \Ric(K) \geq \kappa\;.
\end{align*}
\end{definition}

\begin{remark}\label{rem:all-geodesics}
  Instead of requiring convexity along all geodesics it will be shown
  to be equivalent to require that every pair of densities $\rho_0,
  \rho_1 \in \PX$ can be joined by a constant speed geodesic along
  which the entropy is $\kappa$-convex. Another equivalent condition
  would be to impose a lower bound on the Hessian of $\cH$ in the
  interior $\PXs$ (see Theorem \ref{thm:Ric-equiv} below for the details).
\end{remark}

One of the main contributions of this paper is a tensorisation result for \nl Ricci curvature, that we will now describe.
For $1 \leq i \leq n$, let $K_i$ be an irreducible and reversible
Markov kernel on a finite set $\cX_i$, and let $\pi_i$ denote the
corresponding invariant probability measure. Let $K_{(i)}$ denote the
lift of $K_i$ to the product space $\cX = \cX_1 \times \ldots \times
\cX_n$, defined for $\xx = (x_1, \ldots, x_n)$ and $\yy =
(y_1, \ldots, y_n)$ by
\begin{align*}
  K_{(i)}(\xx, \yy) =  \left\{ \begin{array}{ll}
  K_i(x_i, y_i),
   & \text{if } x_j = y_j \text{ for all } j \neq i,\\
  0,
   & \text{otherwise}.\end{array} \right.
\end{align*}
For a sequence $\{\alpha_i\}_{1 \leq i \leq n}$ of nonnegative numbers
with $\sum_{i = 1}^n \alpha_i = 1$, we consider the weighted product
chain, determined by the kernel
\begin{align*}
 K_\alpha := \sum_{i=1}^n \alpha_i K_{(i)}\;.
\end{align*}
Its reversible probability measure is the product measure $\pi = \pi_1 \otimes \cdots \otimes \pi_n$.

\begin{theorem}[Tensorisation of Ricci bounds]\label{thm:main-tensor}
  Assume that $\Ric(K_i)\geq\kappa_i$ for $i=1,\dots,n$. Then we have
  \begin{align*}
    \Ric(K_\alpha)~\geq~\min\limits_i\alpha_i\kappa_i\ .
  \end{align*}
\end{theorem}

Tensorisation results have also been obtained for other notions of Ricci curvature, including the ones by Lott-Sturm-Villani \cite[Proposition 4.16] {S06} and Ollivier \cite[Proposition 27]{Oll09}.
In both cases the proof does not extend to our setting, and completely different ideas are needed here.

As a consequence we obtain a lower bound on the \nl Ricci curvature for
(the kernel $K_n$ of the simple random walk on) the discrete hypercube
$\{0, 1\}^n$, which turns out to be optimal.

\begin{corollary}\label{cor:main-cube}
For $n \geq 1$ we have $\Ric(K_n) \geq \frac2n$.
\end{corollary}

The hypercube is a fundamental building block for applications in
mathematical physics and theoretical computer science, and the problem
of proving ``displacement convexity'' on this space has been an open
problem that motivated the recent paper by Ollivier and Villani
\cite{OV10}, in which a Brunn-Minkowski inequality has been obtained.

\medskip

Another aspect that we wish to single out at this stage, is the fact
that Ricci bounds imply a number of functional inequalities, that are natural discrete counterparts to powerful inequalities in a continuous setting. In particular, we obtain discrete counterparts to the results by Bakry--\'Emery \cite{BE85} and Otto--Villani \cite{OV00}.

To state the results we consider the Dirichlet form
\begin{align*}
  \cE(\phi, \psi)~=~\frac{1}{2}\sum\limits_{x,y\in\cX} \big(\phi(x)-\phi(y)\big) \big(\psi(x)-\psi(y)\big) K(x,y)\pi(x)\,
\end{align*}
defined for functions $\phi, \psi : \cX \to \R$.
Furthermore, we consider the functional
\begin{align*}
  \cI(\rho)~=~
\cE(\rho, \log \rho)  
\end{align*}
defined for $\rho \in \PX$, with the convention that $\cI(\rho) = +
\infty$ if $\rho$ does not belong to $\PXs$.  Its significance here is
due to the fact that it is the time-derivative of the entropy along
the heat flow: $\ddt \cH(P_t \rho) = - \cI(P_t \rho)$. In this sense,
$\cI$ can be regarded as a discrete version of the Fisher information.

\begin{theorem}[Functional inequalities]\label{thm:main-inequalities}
Let $K$ be an irreducible and reversible Markov kernel on a finite set $\cX$.
\begin{enumerate}
\item If $\Ric(K) \geq \kappa$ for some $\kappa \in \R$, then the
  H$\cW$I-inequality
\begin{align}
  \tag{H$\cW$I($\kappa$)}
  \label{eq:HcWI}
  \cH(\rho)~\leq~\cW(\rho,\one)\sqrt{\cI(\rho)}-\frac{\kappa}{2}\cW(\rho,\one)^2
\end{align}
holds for all $\rho \in \PX$.
\item If $\Ric(K) \geq \lambda$ for some $\lambda > 0$, then the
  modified logarithmic Sobolev inequality
\begin{align}
  \tag{MLSI($\lambda$)}
  \label{eq:ModLSI}
    \cH(\rho)~\leq~\frac{1}{2\lambda}\cI(\rho)
\end{align}
holds for all $\rho \in \PX$.
\item If $K$ satisfies $(\mLSI(\lambda))$ for some $\lambda > 0$, then
  the modified Talagrand inequality
  \begin{align}
    \tag{T$_\cW$($\lambda$)}
  \label{eq:Tal}
      \cW(\rho,\one)~\leq~\sqrt{\frac{2}{\lambda}\cH(\rho)}
  \end{align}
holds for all $\rho\in\PX$.
\item If $K$ satisfies $(\mTal(\lambda))$ for some $\lambda > 0$, then
  the Poincar\'e inequality
  \begin{align}
    \tag{{P($\lambda$)}}
  \label{eq:Poinc}
         \left\|{\phi}\right\|^2_{L^2(\cX,\pi)}~\leq~
         \frac{1}{\lambda}\cE(\phi, \phi)
  \end{align}
holds for all functions $\psi : \cX \to \R$.
\end{enumerate}
\end{theorem}

Here, $\one$ denotes the density of the stationary measure $\pi$.

The first inequality in Theorem \ref{thm:main-inequalities} is a discrete counterpart to the HWI-inequality
from Otto and Villani \cite{OV00}, with the difference that the metric
$W_2$ has been replaced by $\cW$. 

The second result is as a discrete version of the celebrated criterion
by Bakry-{\'E}mery \cite{BE85}, who proved the corresponding result on
Riemannian manifolds. Classically, the Bakry-{\'E}mery criterion applies to weighted Riemannian manifolds $(\mathcal{M}, e^{-V}{\rm vol}_\mathcal{M})$, and asks for a lower bound on the generalised Ricci curvature given by $\Ric_{\mathcal{M}} + \Hess V$. As in our setting we allow for general $K$ and $\pi$, the potential $V$ is already incorporated in $K$ and $\pi$, and our notion of Ricci curvature could be thought of as the analogue of this generalised Ricci curvature.

The modified logarithmic Sobolev inequality $(\mLSI)$ 
is motivated by the fact that it yields an explicit rate of exponential decay of the entropy along the heat flow. It has been  extensively studied (see, e.g, \cite{BT06,CDPP09}), along with different discrete logarithmic Sobolev inequalities in the literature (e.g., \cite{AL00,BL98}).

The third part is a discrete counterpart to a famous result by Otto
and Villani \cite{OV00}, who showed that the logarithmic Sobolev
inequality implies the so-called $T_2$-inequality; recall that the
$T_p$-inequality is the analogue of $\mTal$, in which $\cW$ is
replaced by $W_p$, for $1 \leq p < \infty$. These inequalities have
been extensively studied in recent years. We refer to \cite{GL10} for
a survey and to \cite{SaTe09} for a study of the $T_1$-inequality in a discrete setting.

The modified Talagrand inequality $\mTal$ that we
consider is new. This inequality combines some of the good
properties of $T_1$ and $T_2$, as we shall now discuss.

Like $T_1$, it is weak enough to be applicable in a discrete
setting. In fact, we shall prove that $\mTal(\lambda)$ holds on the
discrete hypercube $\{0,1\}^n$ with the optimal constant $\lambda =
\frac2n$. By contrast, the $T_2$-inequality does not even hold on the
two-point space, and it has been an open problem to find an adequate
substitute.

Like $T_2$, and unlike $T_1$, $\mTal$ is strong enough to capture spectral information. In fact, the fourth part in Theorem \ref{thm:main-inequalities} asserts that it implies a
Poincar\'e inequality with constant $\lambda$.

Furthermore, we shall show that $\mTal$ yields good bounds on
the sub-Gaussian constant, in the sense that
\begin{align}\label{eq:intro-sub-Gauss}
   \E_\pi \big[e^{t(\phi - \E_\pi[\phi])}\big] \leq \exp\Big(\frac{t^2}{4\lambda}\Big)
\end{align}
for all $t > 0$ and all functions $\phi: \cX \to \R$ that are Lipschitz constant $1$
with respect to the graph norm. Here, we use the notation
$\E_\pi[\phi] = \sum_{x \in \cX} \phi(x)\pi(x)$. As is well known,
this estimate yields the concentration inequality
\begin{align*}
 \pi\big( \phi -  \E_\pi[\phi] \geq h\big) \leq e^{-\lambda h^2}
\end{align*}
for all $h > 0$. The proof of \eqref{eq:intro-sub-Gauss} relies on the
fact, proved in Section \ref{sec:metric}, that the metric $\cW$ can be
bounded from below by $W_1$ (with respect to the graph metric), so
that $\mTal(\lambda)$ implies a $T_1(2\lambda)$-inequality, which is known to be equivalent to the sub-Gaussian inequality \cite{BG99}.
 
The proof of Theorem \ref{thm:main-inequalities} follows the approach
by Otto and Villani. On a technical level, the proofs are simpler in
the discrete case, since heuristic arguments from Otto and Villani
are essentially rigorous proofs in our setting, and no additional PDE
arguments are required as in \cite{OV00}.

To summarise we have the following sequence of implications, for any
$\lambda > 0$:
\begin{align*}
  \Ric(K)\geq\lambda
   \ \Rightarrow \
  \text{\ref{eq:MLSI}}
    \ \Rightarrow \
   \text{\ref{eq:T}}
     \ \Rightarrow \
 \left\{\begin{array}{l}
   \text{\ref{eq:P}}\\
\mathrm{T}_1(2\lambda)\ .\end{array}\right.
\end{align*}

\subsection*{Other notions of Ricci curvature}
This is of course not the first time that a notion of Ricci curvature
has been introduced for discrete spaces, but the notion considered
here appears to be the closest in spirit to the one by
Lott-Sturm-Villani. Furthermore it seems to be the first that yields
natural analogues of the results by Bakry--\'Emery and
Otto--Villani.

A different notion of Ricci curvature has been introduced by Ollivier
\cite{Oll07,Oll09}. This notion is also based on ideas from optimal
transport, and uses the $L^1$-Wasserstein metric $W_1$, which
behaves better in a discrete setting than $W_2$. Ollivier's criterion
has the advantage of being easy to check in many examples. Furthermore, in some interesting cases it yields functional inequalities with good -- yet non-optimal -- constants. Moreover, Ollivier does not assume reversibility, whereas this is strongly used in our approach.
It is not completely clear how Ollivier's notion relates to the one by Lott-Sturm-Villani (see \cite{OV10} for a discussion). Furthermore, it does not seem to be directly comparable to the concept studied here, as it relies on a metric on the underlying space, which is not the case in our approach.

In the setting of graphs Ollivier's Ricci curvature has been further studied in the recent preprints
\cite{BaJoLi11,HuJoLi11,JoLi11}. 

Another approach has been taken by Lin and Yau \cite{LY10}, who
defined Ricci curvature in terms the heat semigroup.

Bonciocat and Sturm \cite{BS09} followed a different approach to
modify the Lott-Sturm-Villani criterion, in which they circumvented 
the lack of midpoints in the $W_2$-Wasserstein metric by allowing for
approximate midpoints. A Brunn-Minkowski inequality in this spirit has been proved on the discrete hypercube by Ollivier and Villani \cite{OV10}.

\subsection*{Organisation of the paper}
In Section \ref{sec:metric} we collect basic properties of the metric
$\cW$ and formulate an equivalent definition, that is more convenient
to work with in some situations. Geodesics in the $\cW$-metric are
studied in Section \ref{sec:geodesics}. In particular it is shown that
every pair of densities can be joined by a constant speed geodesic. In
Section \ref{sec:Ricci} we present the definition of \nl Ricci curvature
and give a characterisation in terms of the Hessian of the
entropy. Section \ref{sec:examples} contains a criterion that allows
to give lower bounds on the Ricci curvature in some basic examples,
including the discrete circle and the discrete hypercube. A
tensorisation result is contained in Section
\ref{sec:tensorisation}. Finally, we introduce new versions of
well-known functional inequalities in Section \ref{sec:inequalities}
and prove implications between these and known inequalities.

\subsection*{Note added} 
After essentially finishing this paper, the authors have been informed
about the preprint \cite{Mie11b}, in which geodesic convexity of the
entropy for Markov chains has been studied as well. The results
obtained in both papers do not overlap significantly and have been
obtained independently.

\subsection*{Acknowledgement} 
{\small
The authors are grateful to Nicola Gigli and Karl-Theodor Sturm for stimulating discussions on this paper and related topics. They thank the anonymous referees for detailed comments and valuable suggestions.}

\section{The metric \texorpdfstring{$\cW$}{W}}\label{sec:metric}

In this section we shall study some basic properties of the metric
$\cW .$ Throughout we shall work with an irreducible and reversible
Markov kernel $K$ on a finite set $\cX$. The unique steady state will
be denoted by $\pi$, and we shall write $P_t := e^{t(K-I)}$, $t
\geq0$, to denote the corresponding Markov semigroup.

We start by introducing some notation.

\subsection{Notation}\label{sec:notation}

For $\phi \in \R^\cX$ we consider the \emph{discrete gradient} $\nabla \phi
\in \R^{\cX \times \cX}$ defined by
\begin{align*}
 \nabla \phi(x,y) := \phi(y) - \phi(x)\;.
\end{align*}
For $\Psi \in \R^{\cX \times \cX}$ we consider the \emph{discrete
  divergence} $\nabla \cdot \Psi \in \R^\cX$ defined by
\begin{align*}
( \nabla \cdot \Psi )(x) 
  := \frac12 \sum_{y \in \cX}  (\Psi(x,y) - \Psi(y,x) ) K(x,y) \in \R\;.
\end{align*}
With this notation we have
\begin{align*}
\Delta :=  \nabla \cdot \nabla    = K - I\;,
\end{align*}
and the integration by parts formula
\begin{align*}
 \ip{\nabla \psi, \Psi}_\pi = -\ip{\psi,\nabla\cdot \Psi}_\pi
\end{align*}
holds. Here we write, for $\phi,\psi \in \R^\cX$ and $\Phi,\Psi \in \R^{\cX \times \cX}$,
\begin{align*}
\ip{\phi, \psi}_\pi &= \sum_{x \in \cX} \phi(x) \psi(x) \pi(x)\;, \\
\ip{\Phi, \Psi}_\pi &= \frac12 \sum_{x,y \in \cX} 
         \Phi(x,y) \Psi(x,y) K(x,y) \pi(x)\;.
\end{align*}

From now on we shall fix a function $\theta : \R_+
  \times \R_+ \to \R_+$ 
satisfying the following assumptions:

\begin{assumption} \label{ass:theta} The function $\theta$ has the following properties:
\begin{itemize}
\item[(A1)] (Regularity): $\theta$ is continuous on $\R_+ \times \R_+$ and $C^\infty$ on $(0,\infty) \times (0,\infty)$;
\item[(A2)] (Symmetry): $\theta(s,t) = \theta(t,s)$ for $s, t \geq 0$;
\item[(A3)] (Positivity, normalisation): $\theta(s,t) > 0$ for $s,t > 0$ and $\theta(1,1)=1$;
\item[(A4)] (Zero at the boundary): $\theta(0,t) = 0$ for all $t \geq  0$;
\item[(A5)] (Monotonicity): $\theta(r, t) \leq \theta(s,t)$ for all $0
  \leq r \leq s$ and $t \geq 0$;
\item[(A6)] (Positive homogeneity): $\theta(\lambda s, \lambda
    t) = \lambda \theta(s,t)$ for $\lambda > 0$ and $s,t \geq 0$;
\item[(A7)] (Concavity): 
the function $\theta : \R_+ \times \R_+ \to \R_+$  is concave.
\end{itemize}
\end{assumption}
It is easily checked that these assumptions imply that $\theta$ is
bounded from above by the arithmetic mean :
\begin{align}\label{eq:theta-arithmetic-ineq}
  \theta(s,t)~\leq~\frac{s+t}{2}\qquad\forall s,t\geq 0\ .
\end{align}
In the next result we collect some properties of the function $\theta$,
which turn out be very useful to obtain \nl Ricci curvature bounds.

\begin{lemma}\label{lem:thetatricks}
For all $s,t,u,v>0$ we have
\begin{align}\label{eq:tricktheta1}
  s\cdot\partial_1\theta(s,t) + t\cdot\partial_2\theta(s,t)~ &=~\theta(s,t)\ ,\\\label{eq:4term}
  s\cdot\partial_1\theta(u,v) + t\cdot\partial_2\theta(u,v) - \theta(s,t)~&\geq~0\ .
\end{align}
\end{lemma}

\begin{proof}
  The equality \eqref{eq:tricktheta1} follows immediately from the
  homogeneity (A6) by noting that the left hand side equals
  $\ddr\big\vert_{r=1}\theta(rs,rt)$.  Let us prove
  \eqref{eq:4term}. Note that by the concavity (A7) of $\theta$ the
  gradient $\nabla\theta$ is a monotone operator from $\R_+^2$ to
  $\R^2$. Hence, for all $s,t,x,y>0$ we have
  \begin{align*}
    (s-x)\Big(\partial_1\theta(s,t)-\partial_1\theta(x,y)\Big) + (t-y)\Big(\partial_2\theta(s,t)-\partial_2\theta(x,y)\Big)~\leq~0\ .
  \end{align*}
  By the homogeneity (A6) both $\partial_1\theta$ and
  $\partial_2\theta$ are $0$-homogeneous. Taking now in particular
  $x=\eps u, y=\eps v$ and letting $\eps\to 0$ we obtain
  \begin{align*}
    s\Big(\partial_1\theta(s,t)-\partial_1\theta(u,v)\Big) + t\Big(\partial_2\theta(s,t)-\partial_2\theta(u,v)\Big)~\leq~0\ .
  \end{align*}
  From this we deduce \eqref{eq:4term} by an application of
  \eqref{eq:tricktheta1}.
\end{proof}

The most important example for our purposes is the logarithmic mean
defined by
\begin{align*}
 \theta(s,t) := \int_0^1 s^{1-p} t^{p} \dd p
  			  = \frac{s-t}{\log s - \log t}\;,
\end{align*}
the latter expression being valid if $s,t > 0$ and $s \neq t$.
For $\rho \in \PX$ and $x , y \in \cX$ we define
\begin{align*}
\hrho(x,y)=\theta(\rho(x),\rho(y))\;.
\end{align*}

For a fixed $\rho \in \PX$ it will be useful to consider the
Hilbert space $\cG_\rho$ consisting of all (equivalence classes of)
functions $\Psi : \cX \times \cX \to \R$, endowed with the inner
product
\begin{align}\label{eq:ip}
 \ip{\Phi, \Psi}_\rho := 
 \frac12 
  \sum_{x,y\in \cX} {\Phi(x,y)}{\Psi(x,y)}{\hrho(x,y)}	 K(x,y)  \pi(x)\;.
\end{align}
Here we identify functions that coincide on the set
$\{(x,y)\in\cX\times\cX:\ \hrho(x,y)K(x,y)>0\}$.  The operator
$\nabla$ can then be considered as a linear operator $\nabla :
L^2(\cX) \to \cG_\rho$, whose negative adjoint is the
\emph{$\rho$-divergence operator} $(\nabla_\rho \cdot) : \cG_\rho \to
L^2(\cX)$ given by
\begin{align*}
( \nabla_\rho \cdot \Psi )(x) 
  := \frac12 \sum_{y \in \cX} (\Psi(x,y) - \Psi(y,x)) \hrho(x,y) K(x,y)\;.
\end{align*}

\subsection{Equivalent definitions of the metric $\cW$}

We shall now state the definition of the metric $\cW$ as defined in
\cite{Ma11}. Here and in the rest of the paper we will use the
shorthand notation
\begin{align*}
 \cA(\rho,\psi) := 
  \| \nabla \psi\|_\rho^2
  = \frac12\sum_{x,y\in \cX} (\psi(y)-\psi(x))^2\hrho(x,y) K(x,y) \pi(x)\ .
\end{align*}
$\rho\in\PX$ and $\psi \in \R^{\cX}$.
\begin{definition}\label{def:metric}
For $\bar\rho_0, \bar\rho_1 \in \PX$ we define 
\begin{align*}
  \cW(\bar\rho_0, \bar\rho_1)^2 := 
  \inf \bigg\{ \int_0^1 \cA(\rho_t, \psi_t) \dd t \ : \ {(\rho, \psi) \in \CE_1(\bar\rho_0,\bar\rho_1)} \bigg\}\;,
\end{align*}
where for $T>0$, $\CE_T(\bar\rho_0,\bar\rho_1)$ denotes the
collection of pairs $(\rho,\psi)$ satisfying the following conditions:
\begin{align} \label{eq:conditions} 
 \left\{ \begin{array}{ll}
{(i)} & \rho : [0,T] \to \R^\cX  \text{ is }C^\infty\;;\\ 
{(ii)} &  \rho_0 = \bar\rho_0\;, \qquad \rho_T = \bar\rho_1\;; \\
{(iii)} &  \rho_t \in \PX \text{ for all $t \in [0,T]$}\;;\\
{(iv)} & \psi  : [0,T] \to \R^{\cX}  \text{ is measurable}\;;\\ 
{(v)} &  \text{For all $x \in \cX$ and all $t\in (0,T)$ we have}\\
       &\displaystyle{\dot \rho_t(x) 
         + \sum_{y \in \cX}
         \big(\psi_t(y) - \psi_t(x)\big)\hrho_t(x,y) K(x,y) = 0}\;.
\end{array} \right.
\end{align}
\end{definition}
Using the notation introduced above, the continuity equation in (v) can be written as
\begin{align}\label{eq:ce-symb}
 \dot \rho_t + \nabla \cdot (\hrho \nabla \psi) = 0\;.
\end{align}
Definition \ref{def:metric} is the same as the one in \cite{Ma11},
except that slightly different regularity conditions have been imposed on
$\rho$. We shall shortly see that both definitions are equivalent.

The following results on the metric $\cW$ have been proved in \cite{Ma11}.

\begin{theorem}\label{thm:ma11-main}
\begin{enumerate}
\item The space $(\PX, \cW)$ is a complete metric space,
  compatible with the Euclidean topology.
\item The restriction of $\cW$ to $\PXs$ is the Riemannian
  distance induced by the following Riemannian structure:
\begin{itemize}
\item the tangent space of $\rho \in \PXs$ can be identified
  with the set
\begin{align*}
 T_\rho := \{ \nabla \psi \ : \ \psi \in \R^\cX \}
\end{align*}
by means of the following identification: given a smooth curve
$(-\eps, \eps) \ni t \mapsto \rho_t \in \PXs$ with $\rho_0 = \rho$, there exists a
unique element $\nabla \psi_0 \in T_\rho$, such that the continuity
equation \eqref{eq:conditions}(v) holds at $t = 0$.
\item The Riemannian metric on $T_\rho$ is given by the inner product
\begin{align*}
 \ip{\nabla \phi, \nabla \psi}_\rho 
   = \frac12 \sum_{x,y\in \cX} 
 (\phi(x) - \phi(y))(\psi(x) - \psi(y)) \hrho(x,y)K(x,y) \pi(x)\;.   
\end{align*}
\end{itemize}
\item \label{item:gradFlow} If $\theta$ is the logarithmic mean, i.e., $\theta(s,t) = \int_0^1 s^{1-p} t^p \dd p$, then the heat flow is the gradient flow of the
  entropy, in the sense that for any $\rho \in \PX$ and $t > 0$,
  we have $\rho_t := P_t \rho \in \PXs$ and
\begin{align}\label{eq:gradFlow}
  D_t \rho_t =  -  \grad \cH(\rho_t)\;.
\end{align} 
\end{enumerate}
\end{theorem}

\begin{remark}\label{rem:t-is-0}
  If $\rho$ belongs to $\PXs$, then the gradient flow equation
  \eqref{eq:gradFlow} holds also for $t = 0$.
\end{remark}

\begin{remark}\label{rem:log-mean}
The relevance of the logarithmic mean can be seen as follows. The heat equation $\dot\rho_t = \Delta \rho_t = \nabla \cdot (\nabla \rho_t)$ can be rewritten as a continuity equation \eqref{eq:ce-symb} provided that 
\begin{align*}
\nabla \psi = - \frac{\nabla \rho}{\hrho}\;.
\end{align*}
On the other hand, an easy computation (see \cite[Proposition 4.2 and Corollary 4.3]{Ma11}) shows that under the identification above, the gradient of the entropy is given by 
\begin{align*}
 \grad_\cW \cH(\rho) = \nabla \log \rho\;.
\end{align*}
Combining these observations, we infer that the heat flow is the gradient flow of the entropy with respect to $\cW$, precisely when 
\begin{align*}
 \frac{\nabla \rho}{\hrho}=  \nabla \log \rho\;,
\end{align*}
i.e., when $\theta$ is the logarithmic mean.

This argument shows that the same heat flow can also be identified as the gradient flow of the functional $\cF(\rho) = \sum_{x\in \cX} f(\rho(x)) \pi(x)$ for any smooth function $f : \R \to \R$ with $f'' > 0$, if one replaces the logarithmic mean by $\theta(r,s) = \frac{r-s}{f'(r) - f'(s)}$. We refer to \cite{Ma11} for the details.
\end{remark}

Our next aim is to provide an equivalent formulation of the definition
of $\cW$, which may seem less intuitive at first sight, but offers
several technical advantages. First, the continuity equation becomes
linear in $V$ and $\rho$, which allows to exploit the concavity of
$\theta$. Second, this formulation is more stable so that we can prove
existence of minimizers in the class $\CE'_0(\bar\rho_0,\bar\rho_1)$.
Similar ideas have already been developed in a continuous setting in
\cite{DNS09}, where a general class of transportation metrics is
constructed based on the usual continuity equation in $\R^n$.

An important role will be played by the function $\alpha : \R \times
\R_+^2 \to \R \cup \{ + \infty\}$ defined by
\begin{align*}
 \alpha(x,s,t) =  \left\{ \begin{array}{ll}
 0\;,
  &  \theta(s,t) = 0\text{ and } x = 0 \;,\\
 \frac{x^2}{\theta(s,t)}\;,
  & \theta(s,t) \neq 0\;,\\
 + \infty\;,
  & \theta(s,t) = 0 \text{ and } x \neq 0\;.\end{array} \right.
\end{align*}
The following observation will be useful.
\begin{lemma}\label{lem:convex}
The function $\alpha$ is lower semicontinuous and convex.
\end{lemma}
\begin{proof}
  This is easily checked using (A7) and the convexity of the function
  $(x,y)\mapsto \frac{x^2}{y}$ on $\R \times (0,\infty)$.
\end{proof}

Given $\rho \in \PX$ and $V \in \R^{\cX \times \cX}$  we define
\begin{align*}
 \cA'(\rho, V) := \frac12
 \sum_{x,y\in \cX} \alpha(V(x,y),\rho(x), \rho(y)) K(x,y) \pi(x)\ ,
\end{align*}
and we set
\begin{align*}
 \CE_T'(\bar\rho_0, \bar\rho_1) 
   := \{ (\rho, \psi)  \ : \ (i'), (ii), (iii), (iv'), (v') \text{ hold\,}
            \}\;,
\end{align*}
where 
\begin{align} \label{eq:conditions-2} 
 \left\{ \begin{array}{ll}
{(i')} & \rho : [0,T] \to \R^\cX  \text{ is continuous} \;;\\ 
{(iv')} & V  : [0,T] \to \R^{\cX\times\cX} \text{ is locally integrable}\;;\\
{(v')} & \text{For all $x \in \cX$ we have in the sense of distributions}\\
       &\displaystyle{\dot \rho_t(x) 
         + \frac12 \sum_{y \in \cX}
         \big(V_t(x,y) - V_t(y,x)\big) K(x,y) = 0}\;.\
\end{array} \right.
\end{align}
The continuity equation in $(v')$ can equivalently be written as 
\begin{align*}
 \dot \rho_t + \nabla \cdot V = 0\;. 
\end{align*}
As an immediate consequence of Lemma \ref{lem:convex} we obtain the
following convexity of $\cA'$.
\begin{corollary}\label{cor:convexaction}
  Let $\rho^i\in\PX$ and $V^i\in\R^{\cX\times\cX}$ for $i=0,1$. For
  $\tau\in[0,1]$ set $\rho^\tau:=(1-\tau)\rho^0+ \tau\rho^1$ and
  $V^\tau:=(1-\tau) V^0+ \tau V^1$. Then we have
  \begin{align*}
    \cA'(\rho^\tau,V^\tau)~\leq~(1-\tau)\cA'(\rho^0,V^0)+ \tau
    \cA'(\rho^1,V^1)\ .
  \end{align*}
\end{corollary}
Now we have the following reformulation of Definition
\ref{def:metric}.
\begin{lemma}\label{lem:smooth-reformulation}
  For $\bar\rho_0, \bar\rho_1 \in \PX$ we have
  \begin{align*}
    \cW(\bar\rho_0, \bar\rho_1)^2 =   \inf \bigg\{
    \int_0^1 \cA'(\rho_t,V_t) \dd t \ : \ {(\rho, V) \in \CE'_1(\bar\rho_0, \bar\rho_1)} \bigg\}\;.
  \end{align*}
  Furthermore, if $\bar\rho_0,\bar\rho_1 \in \PXs$, condition $(iv)$ in
  \eqref{eq:conditions} can be reinforced into: ``$\psi : [0,T] \to
  \R^{\cX}$ is $C^\infty$''.
\end{lemma}
\begin{proof}
  The inequality ``$\geq$'' follows easily by noting that the infimum is
  taken over a larger set. Indeed, given a pair
  $(\rho,\psi)\in\CE_1(\bar\rho_0,\bar\rho_1)$ we obtain a pair
  $(\rho,V)\in\CE'_1(\bar\rho_0,\bar\rho_1)$ by setting
  $V_t(x,y)=\nabla\psi_t(x,y)\hrho_t(x,y)$ and we have
  $\cA'(\rho_t,V_t)=\cA(\rho_t,\psi_t)$.

  To show the opposite inequality ``$\leq$'', we fix an arbitrary pair
  $(\rho, V) \in \CE_1'(\bar\rho_0,\bar\rho_1)$. It is sufficient to
  show that for every $\eps>0$ there exists a pair
  $(\rho^\eps,\psi^\eps)\in\CE_1(\bar\rho_0,\bar\rho_1)$ such that
\begin{align*}
\int_0^1\cA(\rho^\eps_t,\psi^\eps_t)\dd t~\leq~\int_0^1\cA'(\rho_t,V_t)\dd t + \eps\ .
\end{align*}
  For this purpose we first regularise $(\rho,V)$ by a mollification
  argument.
  We thus define $(\tilde\rho,\tilde
  V):[-\eps,1+\eps]\to\PX\times\R^{\cX\times\cX}$ by
  \begin{align*}
    (\tilde\rho_t,\tilde V_t)~=~
    \begin{cases}
      (\rho(0),0)\ , & t\in[-\eps,\eps)\ ,\\
      (\rho(\frac{t-\eps}{1-2\eps}),\frac{1}{1-2\eps}V(\frac{t-\eps}{1-2\eps}))\ , & t\in[\eps,1-\eps)\ ,\\
      (\rho(1),0)\ , & t\in[1-\eps,1+\eps]\ ,
    \end{cases}
  \end{align*}
  and take a nonnegative smooth function $\eta:\R\to\R_+$ which
  vanishes outside of $[-\eps,\eps]$, is strictly positive on
  $(-\eps,\eps)$ and satisfies $\int \eta(s)\dd s = 1$.  For
  $t\in[0,1]$ we define
  $$\rho^\eps_t~=~\int\eta(s)\tilde\rho_{t+s}\dd s\ ,\qquad V^\eps_t~=~\int\eta(s)\tilde V_{t+s}\dd s\ .$$
  Now $t\mapsto \rho_t^\eps$ is $C^\infty$ and using the continuity of
  $\rho$ it is easy to check that
  $(\rho^\eps,V^\eps)\in\CE_1'(\bar\rho_0,\bar\rho_1)$. Moreover,
  using the convexity from Corollary \ref{cor:convexaction} we can
  estimate
  \begin{align*}
    \int_0^1\cA'(\rho^\eps_t,V^\eps_t)\dd t~&\leq~\int_0^1\int\eta(s)\cA'(\tilde\rho_{t+s},\tilde V_{t+s})\dd s\dd t\\
    &\leq~\int_{-\eps}^{1+\eps}\cA'(\tilde\rho_{t},\tilde V_{t})\dd
    t~=~\frac{1}{1-2\eps}\int_0^1\cA'(\rho_t,V_t)\dd t\ .
  \end{align*}
  To proceed further, we may assume without loss of generality that
  $V(x,y)=0$ whenever $K(x,y)=0$. The fact that
  $\int_0^1\cA'(\rho_t,V_t)\dd t$ is finite implies that the set $\{ t
  : \hat\rho_t(x,y) = 0 \text{ and } V_t(x,y) \neq 0 \}$ is negligible
  for all $x, y \in \cX$. Taking properties (A3) and (A4) of
  the function $\theta$ into account, this implies that for the convolved
  quantities the corresponding set $\{ t : \hat\rho^\eps_t(x,y) = 0
  \text{ and } V^\eps_t(x,y) \neq 0 \}$ is empty for all
  $x,y\in\cX$. Hence there exists a measurable function
  $\Psi^\eps:[0,1]\to\R^{\cX\times\cX}$ satisfying
  \begin{align}\label{eq:V2Psi}
    V_t^\eps(x,y)~=~\Psi^\eps_t(x,y)\hat\rho_t^\eps(x,y) \qquad\mbox{ for all } x,y\in\cX \mbox{ and all } t\in[0,1]\ .
  \end{align}

  It remains to find a  function $\psi^\eps : [0,1] \to \R^{\cX}$ such
  that 
  $\nabla_{\rho_t^\eps}\cdot \Psi_t^\eps = \nabla_{\rho_t^\eps}\cdot\nabla\psi_t^\eps$. 
  Let $\Proj_\rho$ denote the orthogonal projection in $\cG_\rho$ onto the range of
  $\nabla$. Then there exists a measurable function
  $\psi^\eps:[0,1]\to\R^\cX$ such that
  $\Proj_{\rho^\eps_t}\Psi^\eps_t=\nabla\psi^\eps_t$. The orthogonal
  decomposition
  \begin{align}\label{eq:orthogonal}
    \cG_{\rho^\eps_t}~=~\Ran(\nabla)\oplus^\perp\Ker(\nabla^*_{\rho^\eps_t})
  \end{align}
  implies that
   $\nabla_{\rho_t^\eps}\cdot \Psi_t^\eps = \nabla_{\rho_t^\eps}\cdot\nabla\psi_t^\eps$,
  hence $(\rho^\eps,\psi^\eps)\in\CE_1(\bar\rho_o,\bar\rho_1)$. Using
  the decomposition \eqref{eq:orthogonal} once more, we infer that
  $\ip{\nabla\psi^\eps_t,\nabla\psi^\eps_t}_{\rho^\eps_t}\leq\ip{\Psi^\eps_t,\Psi^\eps_t}_{\rho^\eps_t}$. This
  implies $\cA(\rho^\eps_t,\psi^\eps_t)\leq\cA'(\rho^\eps_t,V^\eps_t)$
  and finishes the proof of the first assertion.

  If $\bar\rho_0$ and $\bar\rho_1$ belong to $\PXs$, one can
  follow the argument in \cite[Lemma 3.30]{Ma11} and construct a curve
  $(\ulrho, \ulV) \in \CE_1'(\bar\rho_0, \bar\rho_1)$ such that
  $\ulrho_t \in \PXs$ for $t \in [0,1]$ and
\begin{align*}
\int_0^1\cA'(\ulrho_t,\ulV_t)\dd t~\leq~\int_0^1\cA'(\rho_t,V_t)\dd t + \eps\;.
\end{align*}
Then one can apply the argument above. In this case, $\rho^\eps_t(x) >
0$ for all $x \in \cX$ and $t \in [0,1]$, and therefore the function
$\Psi^\eps:[0,1]\to\R^{\cX\times\cX}$ is $C^\infty$. Furthermore,
since the orthogonal projection $P_\rho$ depends smoothly on $\rho \in
\PXs$, the function $\psi^\eps : [0,1] \to \R^\cX$ is smooth as
well.
\end{proof}

\begin{remark}\label{rem:defscoincide}
  In \cite{Ma11} the metric $\cW$ has been defined as in Definition
  \ref{def:metric}, with the difference that $(i)$ in
  \eqref{eq:conditions-2} was replaced by ``$\rho:[0,T]\to\PX$ is
  piecewise $C^1$''. Therefore Lemma \ref{lem:smooth-reformulation}
  shows in particular that Definition \ref{def:metric} coincides with
  the original definition of $\cW$ from \cite{Ma11}.
\end{remark}
\subsection{Basic properties of $\cW$}

As an application of Lemma \ref{lem:smooth-reformulation} we shall
prove the following convexity result, which is a discrete counterpart
of the well-known fact that the squared $L^2$-Wasserstein distance
over Euclidean space is convex with respect to linear interpolation
(see, e.g., \cite[Theorem 5.11]{DNS09}).

\begin{proposition}[Convexity of the squared distance]\label{prop:squared-W-cvx}
  For $i, j = 0, 1$, let $\rho_i^j \in \PX$, and for $\tau \in
  [0,1]$ set $\rho_i^\tau := (1-\tau) \rho_i^0 + \tau \rho_i^1$. Then
\begin{align*}
 \cW(\rho_0^\tau,\rho_1^\tau)^2 
   \leq (1- \tau)  \cW(\rho_0^0,\rho_1^0)^2 
 	 +      \tau   \cW(\rho_0^1,\rho_1^1)^2\;.
\end{align*}
\end{proposition}

\begin{proof}
  Let $\eps > 0$. For $j = 0,1$ we may take
   a pair  $(\rho^j,V^j) \in \CE'(\rho_0^j, \rho_1^j)$ with 
  \begin{align*}
   \int_0^1 \cA'(\rho^j_t, V^j_t) \dd t \leq \cW^2(\rho_0^j, \rho_1^j) + \eps 
  \end{align*}
  in view of Lemma \ref{lem:smooth-reformulation}. For $\tau \in [0,1]$ we set
\begin{align*}
\rho_t^\tau := (1-\tau) \rho_t^0 + \tau \rho_t^1\;, \qquad
   V_t^\tau := (1-\tau)    V_t^0 + \tau    V_t^1\;.
\end{align*}
It then follows that $(\rho^\tau, V^\tau) \in \CE_1'(\rho_0^\tau,
\rho_1^\tau)$, hence
\begin{align*}
 \cW(\rho_0^\tau,\rho_1^\tau)^2 
  & \leq \int_0^1 \cA'(\rho_t^\tau, V_t^\tau) \dd t
\\& \leq (1-\tau) \int_0^1 \cA'(\rho_t^0, V_t^0) \dd t
         +  \tau  \int_0^1 \cA'(\rho_t^1, V_t^1) \dd t
\\&    = (1-\tau)  \cW(\rho_0^0,\rho_1^0)^2 
 	     +  \tau  \cW(\rho_0^1,\rho_1^1)^2
	     + \eps\;.
\end{align*}
Since $\eps > 0$ is arbitrary, this completes the proof.
\end{proof}

In this section we compare $\cW$ to some commonly used metrics. A
first result of this type (see \cite[Lemma 3.10]{Ma11}) gives a lower
bound on $\cW$ in terms of the total variation metric
\begin{align*}
  d_{TV}(\rho_0,\rho_1) = \sum_{x \in \cX} \pi(x) |\rho_0(x) - \rho_1(x)|\;.
\end{align*}
Here, more generally, we shall compare $\cW$ to various Wasserstein
distances. Given a metric $d$ on $\cX$ and $1\leq p < \infty$, recall
that the $L^p$-Wasserstein metric $W_{p,d}$ on $\PX$ is defined by
\begin{align} \label{eq:Wasserstein} W_{p,d}(\rho_0, \rho_1) :=
{\inf\bigg\{ \bigg(\sum_{x, y \in \cX} d(x,y)^p q(x,y) \bigg)^{\frac1p}\
\Big| \ q \in \Gamma(\rho_0, \rho_1) \bigg\}}\;,
\end{align}
where $\Gamma(\rho_0, \rho_1)$ denotes the set of all couplings
between $\rho_0$ and $\rho_1$, i.e.,
\begin{align*}
 \Gamma(\rho_0, \rho_1) := \bigg\{ q : \cX \times \cX \to \R_+ \ \Big| \
 &
    \sum_{y \in \cX} q(x,y) = \rho_0(x) \pi(x)\;, \   
 \\&   \sum_{x \in \cX} q(x,y) = \rho_1(y) \pi(y) \bigg\}\;.
\end{align*}
It is well known (see, e.g., \cite[Theorem 4.1]{Vil09}) that the
infimum in \eqref{eq:Wasserstein} is attained; as usual we shall
denote the collection of minimizers by $\Gamma_o(\rho_0,\rho_1)$.

In our setting there are various metrics on $\cX$ that are natural
to consider. In particular, 
\begin{itemize}
\item the \emph{graph distance} $d_g$ with respect to the graph
  stucture on $\cX$ induced by $K$ (i.e., $\{x,y\}$ is an edge iff $K(x,y)>0$).

\item the  metric $d_{\cW}$, that is, the restriction of $\cW$ from $\PX$ to 
  $\cX$
  under the identification of points in $\cX$ with the corresponding
  Dirac masses:
\begin{align*}
 d_\cW(x,y) := \cW\bigg( \frac{\one_{\{x\}} }{\pi(x)} , \frac{\one_{\{y\}} }{\pi(y)} \bigg)\;.
\end{align*}
\end{itemize}
The induced $L^p$-Wasserstein distances will be denoted by $W_{p,g}$
and $W_{p,\cW}$ respectively. 

We shall now prove lower and upper bounds for the metric $\cW$ in
terms of suitable Wasserstein metrics. We start with the lower bounds. Let us remark that, unlike most other results in this paper,   the second inequality in the following result relies on the normalisation $\sum_{y \in \cX} K(x,y) = 1$.

\begin{proposition}[Lower bounds for $\cW$]
\label{prop:Wasserstein-lowerbound}
For all probability densities $\rho_0, \rho_1 \in \PX$ we have
\begin{align}\label{eq:TotVarBound}
  \frac{1}{\sqrt{2}}d_{TV}(\rho_0,\rho_1) 
  \leq \sqrt{2} W_{1,g}(\rho_0, \rho_1) 
  \leq \cW(\rho_0, \rho_1)\;.
\end{align}
\end{proposition} 

\begin{proof}
Note that $d_{tr}\leq d_g$, where $d_{tr}(x,y)=\one_{x\neq y}$ denotes the trivial distance. Therefore, the first bound follows from the fact that $d_{TV}$ is the $L^1$-Wasserstein distance induced by $d_{tr}$ (see \cite[Theorem 1.14]{Vil03}).

In order to prove the second bound, we fix $\eps > 0$, take $\bar\rho_0,
  \bar\rho_1 \in \PX$ and $(\rho,\psi) \in
  \CE_1(\bar\rho_0,\bar\rho_1)$ with
\begin{align*}
 \bigg(\int_0^1 \cA(\rho_t, \psi_t) \dd t \bigg)^{\frac12} \leq \cW(\bar\rho_0, \bar\rho_1 ) + \eps\;.
\end{align*}
Using the continuity equation from \eqref{eq:conditions} we obtain for
any $\phi : \cX \to \R$,
\begin{align*}
  \Big| &\sum_{x \in \cX} \phi(x) (\rho_0(x) - \rho_1(x)) \pi(x)\Big|
  \\&= \bigg|\int_0^1 \sum_{x \in \cX} \phi(x) \dot\rho_t(x) \pi(x)
  \dd t\bigg| \\& = \bigg|\int_0^1 \sum_{x,y \in \cX} \phi(x) \big(
  \psi_t(x) - \psi_t(y) \big) \hrho_t(x,y) K(x,y) \pi(x) \dd t\bigg|
  \\& = \bigg|\int_0^1 \ip{ \nabla \phi, \nabla \psi_t }_{\rho_t} \dd
  t\bigg| \\& \leq \bigg( \int_0^1 \| \nabla\phi\|_{\rho_t}^2\dd t
  \bigg)^{1/2} \bigg( \int_0^1 \| \nabla \psi_t \|_{\rho_t}^2 \dd t
  \bigg)^{1/2} \\& = \bigg( \int_0^1 \|\nabla\phi\|_{\rho_t}^2 \dd t
  \bigg)^{1/2} (  \cW(\bar\rho_0, \bar\rho_1) + \eps) \;.
\end{align*}
Let $[\phi]_{\Lip}$ denote the Lipschitz constant of $\phi$ with
respect to the graph distance $d_g$, i.e.,
\begin{align*}
[\phi]_{\Lip} := \sup_{x \neq y} \frac{|\phi(x) - \phi(y)|}{d_g(x,y)}\;.
\end{align*}
Applying the inequality \eqref{eq:theta-arithmetic-ineq} and using the
fact that $d_g(x,y) = 1$ if $x \neq y$ and $K(x,y) > 0$, we infer that
\begin{align*}
  \| \nabla\phi\|_{\rho_t}^2 &= \frac12 \sum_{x,y \in \cX}
  \big(\phi(x) - \phi(y)\big)^2 K(x,y) \hrho_t(x,y) \pi(x) \\& \leq
  \frac14 [\phi]_{\Lip}^2 \sum_{x,y \in \cX} K(x,y) \big( \rho_t(x) +
  \rho_t(y)\big) \pi(x) \\& = \frac12 [\phi]_{\Lip}^2 \sum_{x\in \cX}
  \rho_t(x) \pi(x) \sum_{y \in \cX}K(x,y) \\& = \frac12
  [\phi]_{\Lip}^2\;.
\end{align*}
The Kantorovich-Rubinstein Theorem (see, e.g., \cite[Theorem
1.14]{Vil03}) yields
\begin{align*}
  W_{1,g}(\bar\rho_0,\bar\rho_1) & = \sup_{\phi : [\phi]_{\Lip} \leq
    1} \Big| \sum_{x \in \cX} \phi(x) (\bar\rho_0(x) - \bar\rho_1(x))
  \pi(x)\Big| \leq \frac{\cW(\bar\rho_0, \bar\rho_1) + \eps}{\sqrt 2}\;,
\end{align*}
which completes the proof, since $\eps > 0$ is arbitrary.
\end{proof}

Before stating the upper bounds, we provide a simple relation between
$d_g$ and $d_\cW$.

\begin{lemma}\label{lem:distances}
For $x,y \in \cX$ we have 
\begin{align*}
d_\cW(x,y) \leq \frac{c}{\sqrt{k}}d_g(x,y)\;,
\end{align*}
where
\begin{align*}
c =  \int_{-1}^{1} \frac{\dd r}{\sqrt{2\theta(1-r,1+r)}} < \infty\;\qquad\text{and}\qquad
k = \min_{(x,y) \ : \ K(x,y) > 0} K(x,y)\;.
\end{align*}
If $\theta$ is the logarithmic mean, then $c\approx 1.56$.
\end{lemma}

\begin{proof}
  Let $\{x_i\}_{i=0}^n$ be a sequence in $\cX$ with $x_0 = x$, $x_n =
  y$ and $K(x_i, x_{i+1}) > 0$ for all $i$.  We shall use the fact,
  proved in \cite[Theorem 2.4]{Ma11}, that the
  $\cW$-distance between two Dirac measures on a two-point space $\{a,
  b\}$ with transition probabilities $K(a,b) = K(b,a) = p$ is equal to
  $\frac{c}{\sqrt{p}}$. The concavity of $\theta$ readily implies that
  $c$ is finite. Furthermore, it follows from \cite[Lemma 3.14]{Ma11}
  and its proof, that for any pair $x, y \in \cX$ with $K(x,y) > 0$,
  one has
\begin{align*}
 \cW\bigg(\frac{\one_{\{x\}} }{\pi(x)}, \frac{\one_{\{y\}} }{\pi(y)}\bigg)
  \leq c  \sqrt{\frac{\max\{\pi(x),\pi(y)\} }{K(x, y)\pi(x)}}
   \leq   \frac{c}{\sqrt{k}}\;.
\end{align*}
Using the triangle inequality for $\cW$ we obtain
\begin{align*}
d_\cW(x,y) 
   & = \cW\bigg( \frac{\one_{\{x\}} }{\pi(x)} , \frac{\one_{\{y\}} }{\pi(y)} \bigg)
  \leq \sum_{i = 0}^{n-1}
    \cW\bigg( \frac{\one_{\{x_i\}} }{\pi(x_i)} , \frac{\one_{\{x_{i+1}\}} }{\pi(x_{i+1})} \bigg)
  \leq  \frac{nc}{\sqrt{k}}\;,
\end{align*}
hence the result follows by taking the infimum over all such sequences
$\{x_i\}_{i=0}^n$.
\end{proof}

Now we turn to upper bounds for $\cW$ in terms of $L^2$-Wasserstein
distances.

\begin{proposition}[Upper bounds for $\cW$] \label{prop:Wasserstein-upperbound}
For all probability densities $\rho_0, \rho_1 \in \PX$ we have
\begin{align}
 \label{eq:Wasserstein-bounds}
  \cW(\rho_0, \rho_1)
   \leq W_{2,\cW}(\rho_0, \rho_1)
   \leq \frac{c}{\sqrt{k}}W_{2,g}(\rho_0,\rho_1) \ ,
\end{align}
where $c$ and $k$ are as in Lemma \ref{lem:distances}.
\end{proposition} 

\begin{proof}
  We shall prove the first bound, the second one being an immediate
  consequence of Lemma \ref{lem:distances}. For this purpose, we fix
  $\bar\rho_0, \bar\rho_1 \in \PX$ and take $q \in
  \Gamma_o(\bar\rho_0, \bar\rho_1)$. For all $u,v \in \cX$, take a
  curve $(\rho^{u,v}, V^{u,v}) \in \CE'\big(\frac{\one_{\{u\}}
  }{\pi(u)} , \frac{\one_{\{v\}} }{\pi(v)}\big)$ with
\begin{align*}
 \int_0^1 \cA'(\rho_t^{u,v}, V_t^{u,v}) \dd t \leq 
   d_\cW(u,v)^2 + \eps\;,
\end{align*}
and consider the convex combination of these curves, weighted
according to the optimal plan $q$, i.e.,
\begin{align*}
  \rho_t := \sum_{u,v \in \cX} q(u,v) \rho_t^{u,v}\;,\qquad
     V_t := \sum_{u,v \in \cX} q(u,v) V_t^{u,v}\;. 
\end{align*}
It then follows that the resulting curve $(\rho, V)$ belongs to
$\CE_1'(\bar\rho_0, \bar\rho_1)$. Using the convexity result from
Lemma \ref{lem:convex} we infer that
\begin{align*}
 \cW(\bar\rho_0, \bar\rho_1)^2 
    \leq \int_0^1 \cA'(\rho_t, V_t) \dd t
  & \leq \sum_{u,v \in \cX} q(u,v) 
			\int_0^1 \cA'(\rho_t^{u,v}, V_t^{u,v}) \dd t
\\&    \leq \sum_{u,v \in \cX}  q(u,v) (  d_\cW(u,v)^2 + \eps)
\\&    =  W_{2,\cW}(\bar\rho_0, \bar\rho_1)^2 + \eps\;.
\end{align*}
which implies the result.
\end{proof}

\section{Geodesics}
\label{sec:geodesics}

In this section we show that the metric space $(\PX,\cW)$ is a
geodesic space, in the sense that any two densities $\rho_0,\rho_1 \in
\PX$ can be connected by a \emph{(constant speed) geodesic}, i.e., a
curve $\gamma : [0,1] \to \PX$ satisfying
\begin{align*}
 \cW(\gamma_s, \gamma_t) = |s-t| \cW(\gamma_0, \gamma_1)\
\end{align*}
for all $0 \leq s, t\leq 1$.

Let us first give an equivalent characterisation of the infimum in
Lemma \ref{lem:smooth-reformulation}, which is invariant
under reparametrisation.

\begin{lemma}\label{lem:reparametrization}
For any $T>0$ and $\bar\rho_0,\bar\rho_1\in\PX$ we have
\begin{equation}\label{eq:reparametrization}
\cW(\bar\rho_0,\bar\rho_1)~=~\inf\left\{\int_0^T\sqrt{\cA'(\rho_t,V_t)}dt\ : \ (\rho,V)\in\CE'_{T}(\bar\rho_0,\bar\rho_1)\right\}\ .
\end{equation}
\end{lemma}

\begin{proof}
  Taking Lemma \ref{lem:smooth-reformulation} into account, this
  follows from a standard re\-pa\-ra\-me\-tri\-sa\-tion argument. See
  \cite[Lemma 1.1.4]{AGS08} or \cite[Theorem 5.4]{DNS09} for details in
  similar situations.
\end{proof}

\begin{theorem}\label{thm:existenceminimizers}
  For all $\bar\rho_0, \bar\rho_1 \in \PX$ the infimum in Lemma
  \ref{lem:smooth-reformulation} is attained by a pair
  $(\rho,V)\in\CE'_1(\bar\rho_0, \bar\rho_1)$ satisfying
  $\cA'(\rho_t,V_t)=\cW(\bar\rho_0, \bar\rho_1)^2$ for
  a.e. $t\in[0,1]$. In particular, the curve $(\rho_t)_{t\in[0,1]}$ is
  a constant speed geodesic.
\end{theorem}
\begin{proof}
  We will show existence of a minimizing curve by a direct
  argument. Let $(\rho^n,V^n) \in \CE_1'(\bar\rho_0, \bar\rho_1)$ be a
  minimizing sequence. Thus we can assume that
  \begin{align*}
    \sup\limits_n \int_0^1\cA'(\rho^n_t,V^n_t)\dd t~<~C
  \end{align*}
  for some finite constant $C$. Without loss of generality we assume
  that $V^n_t(x,y)=0$ when $K(x,y)=0$. For $x,y\in\cX$, define the
  sequence of signed Borel measures $\nu^n_{x,y}$ on $[0,1]$ by
  $\nu^n_{x,y}(\mathrm{d}t):=V^n_t(x,y)\dd t$. For every Borel set
  $B\subset [0,1]$ we can give the following bound on the total
  variation of these measures:
  \begin{align*}
    \norm{\nu^n_{x,y}}(B)~&\leq~\int_B\abs{V^n_t(x,y)}\dd
    t\leq~\sqrt{C'}\int_B \sqrt{\alpha(V^n_t(x,y),\rho^n_t(x),
      \rho^n_t(y))}\dd t\ ,
  \end{align*}
  where we used the fact that $\rho(x)\leq
  \max\{\pi(z)^{-1}:z\in\cX\}=:C' < \infty$ for $\rho\in\PX$. Using
  H\"older's inequality we obtain
  \begin{align}\label{eq:totalvar-estimate}
    \sum\limits_{x,y\in\cX}\norm{\nu^n_{x,y}}(B)K(x,y)\pi(x)~&\leq~
    \sqrt{2C'\Leb(B)}
    \left(\int_0^1\cA'(\rho^n_t,V^n_t)\dd t\right)^\frac12\nonumber\\
    &\leq~\sqrt{2 C C' \Leb(B)}\ .
  \end{align}
  In particular, the total variation of the measures $\nu_{x,y}^n$ is
  bounded uniformly in $n$. Hence we can extract a subsequence (still
  indexed by $n$) such that for all $x,y\in\cX$ the measures
  $\nu^n_{x,y}$ converge weakly* to some finite signed Borel measure
  $\nu_{x,y}$. The estimate \eqref{eq:totalvar-estimate} also shows
  that $\nu_{x,y}$ is absolutely continuous with respect to the
  Lebesgue measure. Thus there exists $V:[0,1]\to\R^{\cX\times\cX}$
  such that $\nu_{x,y}(\mathrm{d} t):=V_t(x,y)\dd t$. We claim that,
  along the same subsequence, $\rho^n$ converges pointwise to a
  function $\rho:[0,1]\to\PX$. Indeed, using the continuity of $t
  \mapsto \rho^n_t$ one derives from the continuity equation $(v')$ in
  \eqref{eq:conditions-2} that for $s\in[0,1]$ and every $x\in\cX$,
  \begin{align}\label{eq:ce-boundary}
    \rho^n_{s}-\rho^n_0~=~
    \frac12\int\limits_0^{s}\sum\limits_{y\in \cX}(V^n_t(y,x) - V^n_t(x,y))K(x,y)\dd t\ .
  \end{align}
  The weak* convergence of $\nu^n_{x,y}$ implies (see
  \cite[Prop. 5.1.10]{AGS08}) the convergence of the right hand side
  of \eqref{eq:ce-boundary}. Since $\rho^n_0=\bar\rho_0$ for all $n$,
  this yields the desired convergence of $\rho^n_s$ for all $s$, and one
  easily checks that $(\rho,V)\in\CE_1'(\rho_0,\rho_1)$. The weak*
  convergence of $\nu_{x,y}^n$ further implies that the measures
  $\rho^n_t(x)\mathrm{d} t$ converge weakly* to $\rho_t(x)\mathrm{d}
  t$. Applying a general result on the lower-semicontinuity of integral 
  functionals (see
  \cite[Thm. 3.4.3]{But89}) and taking into account Lemma
  \ref{lem:convex}, we obtain
  \begin{align*}
    \int_0^1\cA'(\rho_t,V_t)\dd t~\leq~\liminf\limits_n \int_0^1\cA'(\rho^n_t,V^n_t)\dd t~=~\cW(\bar\rho_0,\bar\rho_1)^2\ .
  \end{align*}
  Hence the pair $(\rho,V)$ is a minimizer of the variational problem
  in the definition of $\cW$. Finally, Lemma \ref
  {lem:reparametrization}  yields
  \begin{align*}
   \int_0^1\sqrt{\cA'(\rho_t,V_t)}\dd t~\geq~\cW(\bar\rho_0,\bar\rho_1)~=~\left(\int_0^1\cA'(\rho_t,V_t)\dd t\right)^{\frac{1}{2}}~\;,
  \end{align*}
  which implies that $\cA'(\rho_t,V_t)=W(\bar\rho_0,\bar\rho_1)^2$ for
  a.e. $t\in[0,1]$.
  
  The fact that $(\rho_t)_t$ is a constant speed geodesic follows now
  by another application of Lemma \ref{lem:reparametrization}.
\end{proof}

We shall now give a characterisation of absolutely continuous curves in the
metric space $(\PX,\cW)$ and relate their length to their minimal
action. First we recall some notions from the theory of analysis in
metric spaces. A curve $(\rho_t)_{t\in[0,T]}$ in $\PX$ is called
\emph{absolutely continuous w.r.t. $\cW$} if there exists $m\in
L^1(0,T)$ such that
\begin{align*}
  \cW(\rho_s,\rho_t)~\leq~\int_s^tm(r)\dd r \quad\text{for all}~0\leq s\leq
  t\leq T\ .
\end{align*}
If $(\rho_t)$ is absolutely continuous, then its \emph{metric derivative} 
\begin{equation*}
\abs{\rho_t'}~:=~\lim\limits_{h\to0}\frac{\cW(\rho_{t+h},\rho_t)}{\abs{h}}
\end{equation*}
exists for a.e. $t\in[0,T]$ and satisfies $\abs{\rho_t'}\leq m(t)$ a.e.
(see \cite[Theorem 1.1.2]{AGS08}). 

\begin{proposition}[Metric velocity]\label{prop:metricderivative}
  A curve $(\rho_t)_{t\in[0,T]}$ is absolutely continuous with respect
  to $\cW$ if and only if there exists a measurable function
  $V:[0,T]\to\R^{\cX\times\cX}$ such that $(\rho,V)\in\CE_{T}'(\rho_0, \rho_T)$ and
  \begin{align*}
    \int_0^T\sqrt{\cA'(\rho_t,V_t)}dt~<~\infty\ .
  \end{align*}
  In this case we have $\abs{\rho_t'}^2\leq\cA'(\rho_t,V_t)$ for
  a.e. $t\in[0,T]$ and there exists an a.e. uniquely defined function
  $\tilde{V} : [0,1] \to \R^{\cX \times \cX}$ such that $(\rho,\tilde
  V)\in\CE_{T}'(\rho_0, \rho_T)$ and $\abs{\rho_t'}^2=\cA'(\rho_t,\tilde{V}_t)$ for
  a.e. $t\in[0,T]$.
\end{proposition}

\begin{proof}
  The proof follows from the very same arguments as in
  \cite[Thm. 5.17]{DNS09}. To contruct the velocity field $\tilde V$,
  the curve $\rho$ is approximated by curves $(\rho^n,V^n)$ which are
  piecewise minimizing. The velocity field $\tilde V$ is then defined
  as a subsequential limit of the velocity fields $V^n$. In our case,
  existence of this limit is guaranteed by a compactness argument
  similar to the one in the proof of Theorem
  \ref{thm:existenceminimizers}.
\end{proof}

For later use we state an explicit formula for the geodesic equations in $\cP_*(\cX)$ from \cite[Proposition 3.4]{Ma11}.
Since the interior $\cP_*(\cX)$ of $\cP(\cX)$ is Riemannian by Theorem
\ref{thm:ma11-main}, local existence and uniqueness of geodesics is
guaranteed by standard Riemannian geometry.

\begin{proposition}\label{prop:geodesic-eqns}
  Let $\bar\rho \in \PXs$ and $\bar\psi \in \R^{\cX}$. On a
  sufficiently small time interval around 0, the unique constant speed
  geodesic with $\rho_0 = \bar\rho$ and initial tangent vector 
  $\nabla \psi_0 = \nabla\bar\psi$ satisfies the following equations:
\begin{equation}\begin{aligned} \label{eq:geod-equs}
\begin{cases}
\partial_t \rho_t(x) + 
   \displaystyle\sum_{y \in \cX}  ( \psi_t(y) - \psi_t(x) ) \hrho_t(x,y) K(x,y)  = 0 \;,\\
 \partial_t \psi_t(x)  + \displaystyle\frac12
     \displaystyle\sum_{y \in \cX} \big(  \psi_t(x) -\psi_t(y) \big)^2 
     		 \partial_1\theta(\rho_t(x), \rho_t(y))  K(x,y) = 0 \;.
\end{cases}
\end{aligned}\end{equation}
\end{proposition}

\section{Ricci curvature}
\label{sec:Ricci}

In this section we initiate the study of a notion of Ricci curvature lower
boundedness in the spirit of Lott, Sturm, and Villani
\cite{LV09,S06}. Furthermore, we present a characterisation, which we
shall use to prove Ricci bounds in concrete examples.

As before, we fix an irreducible and reversible Markov kernel $K$ on a
finite set $\cX$ with steady state $\pi$. The associated Markov
semigroup shall be denoted by $(P_t)_{t \geq0}$.

\begin{assumption}\label{ass:log-mean}
Throughout the remainder of the paper we assume that $\theta$ is the logarithmic mean.
\end{assumption}

We are now ready to state the definition, which has already been given in \cite[Definition 1.3]{Ma11}.

\begin{definition}\label{def:Ricci}
  We say that $K$ has \emph{\nl Ricci curvature bounded from below by
    $\kappa \in \R$} and write $\Ric(K) \geq \kappa$, if the following
  holds: for every constant speed geodesic $(\rho_t)_{t \in [0,1]}$
  in $(\PX, \cW)$ we have
\begin{align}\label{eq:K-convex}
  \cH(\rho_t) \leq
  (1-t) \cH(\rho_0) + t \cH(\rho_1) - \frac\kappa{2} t(1-t) \cW(\rho_0, \rho_1)^2\;.
\end{align}
\end{definition}

An important role in our analysis is played by the quantity
$\cB(\rho, \psi) $, which is defined for $\rho \in \PXs$ and $\psi \in
\R^\cX$ by
\begin{equation}\begin{aligned}\label{eq:B}
  \cB(\rho, \psi) 
  := &\ \frac12
  \bip{\hDelta\rho\, \cdot\, \nabla \psi, \nabla \psi }_\pi
    - \bip{\hrho  \cdot\,\nabla \psi\ ,\, \nabla \Delta\psi}_\pi 
\\  = &\ 
  \frac14
       \sum_{x,y,z \in \cX}
         \big(\psi(x) - \psi(y)\big)^2 
        \Big( \partial_1 \theta\big(\rho(x), \rho(y)\big) \big(\rho(z) - \rho(x) \big)K(x,z)
 \\& \qquad
    + \partial_2 \theta\big(\rho(x), \rho(y)\big) \big(\rho(z) - \rho(y) \big)K(y,z)\Big)
       K(x,y) \pi(x)
  \\ &  - 
   \frac12
       \sum_{x,y,z \in \cX}
         \Big( K(x,z) \big( \psi(z) - \psi(x) \big)
           - K(y,z)   \big( \psi(z) -\psi(y)\big)
             \Big)
  \\& \qquad \times       \big(\psi(x) - \psi(y)\big)
        \hrho(x,y)K(x,y)\pi(x)\;,
\end{aligned}\end{equation}
where
\begin{align*}
 \hDelta \rho(x,y) :=  
 \partial_1 \theta(\rho(x), \rho(y)) \Delta \rho(x) 
  + \partial_2 \theta(\rho(x), \rho(y)) \Delta \rho(y)\;.
\end{align*}
The significance of $\cB(\rho, \psi)$ is mainly due to the following
result:

\begin{proposition}\label{prop:Hessian}
For $\rho \in \PXs$ and $\psi \in \R^\cX$ we have 
\begin{align*}
\bip{ \Hess \cH(\rho) \nabla \psi\ ,\, \nabla \psi}_\rho
    = \cB(\rho,\psi)\;.
\end{align*}
\end{proposition}

\begin{proof}
  Take $(\rho , \psi)$ satisfying the geodesic equations
  \eqref{eq:geod-equs}, so that
\begin{align*}
\bip{ \Hess \cH(\rho_t) \nabla \psi_t\ ,\, \nabla \psi_t}_{\rho_t} = \ddtt \cH(\rho_t)\;.
\end{align*}
Using the continuity equation we obtain
\begin{align*}
 \ddt \cH(\rho_t)
    &=  - \bip{ 1 + \log \rho_t , \nabla\cdot( \hat\rho_t \nabla \psi_t) }_\pi
  \\&  =  \bip{ \nabla \log \rho_t ,  \hat\rho_t \cdot \nabla \psi_t }_\pi
  \\&  =  \bip{ \nabla  \rho_t ,  \nabla \psi_t }_\pi\;.
\end{align*}
Furthermore, 
\begin{align*}
  \ddtt \cH(\rho_t)
    & =  \bip{ \nabla \partial_t \rho_t ,  \nabla \psi_t }_\pi
        + \bip{ \nabla  \rho_t ,  \nabla \partial_t \psi_t }_\pi 
   \\&  = -  \bip{  \partial_t \rho_t ,  \Delta \psi_t }_\pi
         -  \bip{ \Delta  \rho_t ,  \partial_t \psi_t }_\pi \ .
\end{align*}
Using the continuity equation we obtain
\begin{align*}
  \bip{  \partial_t \rho_t ,  \Delta \psi_t }_\pi
  & = - \bip{   \nabla\cdot( \hat\rho \nabla \psi_t) ,  \Delta \psi_t }_\pi
 \\&  =   \bip{ \hat\rho_t \nabla \psi_t ,  \nabla\Delta \psi_t }_\pi
   =   \bip{  \nabla \psi_t ,  \nabla\Delta \psi_t }_{\rho_t}\;.
\end{align*}
Furthermore, applying the geodesic equations \eqref{eq:geod-equs} and
the detailed balance equations \eqref{eq:detailed-balance}, we infer
that
\begin{align*}
& \bip{ \Delta  \rho_t ,  \partial_t \psi_t }_\pi 
 \\&  =
   -\frac12
       \sum_{x,y,z \in \cX}
         \big(\psi_t(x) - \psi_t(y)\big)^2 \partial_1 \theta\big(\rho_t(x), \rho_t(y)\big) 
\\& \qquad\qquad\qquad  \times         \big(\rho_t(z)  - \rho_t(x)\big) K(x,y) K(x,z) \pi(x)
 \\&  =
   -\frac14
       \sum_{x,y,z \in \cX}
         \big(\psi_t(x) - \psi_t(y)\big)^2 
        \Big( \partial_1 \theta\big(\rho_t(x), \rho_t(y)\big) \big(\rho_t(z) - \rho_t(x) \big)K(x,z)
 \\& \qquad\qquad\qquad
   + \partial_2 \theta\big(\rho_t(x), \rho_t(y)\big) \big(\rho_t(z) - \rho_t(y) \big)K(y,z)\Big)
       K(x,y)  \pi(x)
 \\& = -\frac12 \bip{\hDelta\rho_t\, \cdot\, \nabla \psi_t, \nabla \psi_t }_\pi\;.
\end{align*}
Combining the latter three identities, we arrive at
\begin{align*}
   \ddtt \cH(\rho_t)
     = - \bip{     \nabla \psi_t ,  \nabla\Delta \psi_t }_{\rho_t}
        +  \frac12 \bip{\hDelta\rho_t\, \cdot\, \nabla \psi_t, \nabla \psi_t }_\pi\;,
\end{align*}
which is the desired identity.
\end{proof}

Our next aim is to show that $\kappa$-convexity of $\cH$ along
geodesics is equivalent to a lower bound of the Hessian of $\cH$ in
$\PXs$.  Since the Riemannian metric on $(\PX, \cW)$ degenerates at
the boundary, this is not an obvious result.
In particular, in order to prove the implication ``$(4) \Rightarrow
(3)$'' below we cannot directly apply the equivalence between the so-called EVI \eqref{eq:EVI} and the usual gradient flow equation, which holds on complete Riemannian
manifolds (see, e.g., \cite[Proposition 23.1]{Vil09}). Therefore, we take a different approach,
based on an argument by Daneri and Savar\'e \cite{DS08}, which avoids
delicate regularity issues for geodesics. An additional benefit of
this approach is that we expect it to apply in a more general setting
where the underlying space $\cX$ is infinite, and finite-dimensional
Riemannian techniques do not apply at all.

\begin{remark}\label{rem:B}
  The quantity $\cB(\rho,\psi)$ arises naturally in the Eulerian
  approach to the Wasserstein metric, as developed in \cite{
  DS08,OW06}.
  In fact, in a crucial argument from 
  \cite{DS08}, the authors consider a certain two-parameter family of
  measures $(\rho_t^s)$ and functions $(\psi_t^s)$ on a Riemannian
  manifold $\cM$, and show that
\begin{align} \label{eq:DS-ineq}
 \partial_s \cH(\rho_t^s)  
 + \frac12 \partial_t \int |\nabla \psi_t^s|^2 \dd  \rho_t^s
   = - B(\rho_t^s, \psi_t^s)\;,
\end{align}
where  
\begin{align*}
B(\rho, \psi)
  := \int_\cM  \Big(  \frac12  \Delta ( |\nabla \psi|^2) 
    - \ip{\nabla \psi, \nabla \Delta\psi} \Big) \dd \rho\;.
\end{align*}
Since Bochner's formula asserts that
\begin{align*}
 B(\rho, \psi)
  :=  \int_\cM   |D^2 \psi|^2
    + {\rm Ric}(\nabla \psi, \nabla\psi)  \dd \rho\;,
\end{align*}
one obtains a lower bound on $B$ if the Ricci curvature is bounded
from below. The lower bound on $B$ can be used to prove an evolution
variational inequality, which in turn yields convexity of the entropy along
$W_2$-geodesics.

In our setting, the quantity $\cB(\rho, \psi)$ can be regarded as a
discrete analogue of $B(\rho, \psi)$. Therefore the inequality
$\cB(\rho, \psi) \geq \kappa \cA(\rho,\psi)$ could be interpreted as a
one-sided Bochner inequality, which allows to adapt the strategy from
\cite{DS08} to the discrete setting.
\end{remark}

In the following result and the rest of the paper we shall use the notation
\begin{align*}
  \ddtr f(t) = \limsup_{h \downarrow 0} \frac{f(t+h) - f(t)}{h}\;.
\end{align*}

\begin{theorem}\label{thm:Ric-equiv}
  Let $\kappa \in \R$. For an irreducible and reversible Markov kernel
  $(\cX, K)$ the following assertions are equivalent:
\begin{enumerate}
\item $\Ric(K) \geq \kappa$\;;
\item For all $\rho, \nu \in \PX$, the following `evolution
  variational inequality' holds for all $t \geq 0$:
\begin{align}\label{eq:EVI}
  \frac{1}{2}\ddtr \cW^2(P_t \rho, \nu) + \frac{\kappa}{2}
\cW^2(P_t \rho, \nu) \leq \cH(\nu) - \cH(P_t \rho)\;;
\end{align}
\item For all $\rho, \nu \in \PXs$, \eqref{eq:EVI} holds for
  all $t \geq 0$;
\item For all $\rho \in \PXs$ and $\psi \in \R^\cX$ we have
\begin{align*}
 \cB(\rho, \psi) \geq \kappa \cA(\rho, \psi)\;.
\end{align*}
\item For all $\rho \in \PXs$ we have 
\begin{align*}
\Hess \cH(\rho) \geq \kappa\;;
\end{align*}
\item For all $\bar\rho_0, \bar\rho_1 \in \PXs$ there exists a
  constant speed geodesic $( \rho_t)_{t \in [0,1]}$ satisfying
  $\rho_0 = \bar\rho_0$, $\rho_1 = \bar\rho_1$, and
  \eqref{eq:K-convex}.
\end{enumerate}

\end{theorem}

\begin{proof}
``$(3) \Rightarrow (2)$'': This is a special case of \cite[Theorem 3.3]{DS08}.

``$(2) \Rightarrow (1)$'': This follows by applying \cite[Theorem
3.2]{DS08} to the metric space $(\PX, \cW)$ and the functional
$\cH$.

``$(1) \Rightarrow (6)$'': This is clear in view of Theorem \ref{thm:existenceminimizers}.

``$(6) \Rightarrow (5)$'': Take $\rho \in \PXs$ and $\psi \in \R^\cX$
and consider the unique solution $(\rho_t, \psi_t)_{t \in (-\eps,
  \eps)}$ to the geodesic equations with $\rho_0 = \rho$ and $\psi_0 =
\psi$ on a sufficiently small time interval around 0. Using the local
uniqueness of geodesics and (6), we infer that
\begin{align*}
\Hess \cH(\rho)(\nabla\psi) = \ddtt\Big|_{t = 0}\cH(\rho_t) \geq \kappa \|\nabla \psi\|_\rho^2
\end{align*}
(see, e.g., the implication ``$(ii)\Leftrightarrow (i)$'' in \cite[Proposition 16.2]{Vil09}).

``$(5) \Rightarrow (4)$'': This follows from Proposition \ref{prop:Hessian}.

``$(4) \Rightarrow (3)$'': We follow \cite{DS08}. In view of Lemma
\ref{lem:smooth-reformulation} we can find a smooth curve $(\rho^{\cdot},
\psi^{\cdot}) \in \CE_1(\nu, \rho)$ satisfying
\begin{align}\label{eq:almost-geod}
 \int_0^1 \cA(\rho^s, \psi^s) \dd s < \cW(\rho, \nu)^2 + \eps\;.
\end{align}
Note in particular that $s \mapsto \rho^s$ and $s \mapsto\psi^s$ are sufficiently regular to
apply Lemma \ref{lem:dan-sav} below. Using the notation from this
lemma, we infer that
\begin{align*}
  \frac12 \partial_t \cA(\rho_t^s, \psi_t^s)
 	 + \partial_s \cH(\rho_t^s) =  - s \cB(\rho_t^s, \psi_t^s)\;.
\end{align*}
Using the assumption that $\cB \geq \kappa \cA$ we infer that
\begin{align*}
 \frac12 \partial_t \Big( e^{2\kappa st} \cA(\rho_t^s, \psi_t^s) \Big)
 	 + \partial_s  \Big( e^{2\kappa st} \cH(\rho_t^s) \Big)
	 \leq 2 \kappa  t e^{2\kappa st} \cH(\rho_t^s)\;.
\end{align*}
Integration with respect to $t \in [0,h]$ and $s \in [0,1]$ yields
\begin{align*}
&  \frac12 \int_0^1 \Big( e^{2\kappa sh} \cA(\rho_h^s, \psi_h^s)
    -  \cA(\rho_0^s, \psi_0^s) \Big) \dd s
\\&\qquad\qquad+  \int_0^h
\Big(  e^{2\kappa t} \cH(\rho_t^1) -  \cH(\rho_t^0) \Big)\dd t
   \leq  2 \kappa  \int_0^1 \int_0^h  t e^{2\kappa st} \cH(\rho_t^s)
   \dd t \dd s\;.
\end{align*}
Arguing as in \cite[Lemma 5.1]{DS08} we infer that
\begin{align*}
 \int_0^1  e^{2\kappa sh} \cA(\rho_h^s, \psi_h^s) \dd s
  \geq m(\kappa h)\cW^2(P_h\rho, \nu)\;,
\end{align*}
where $m(\kappa) = \frac{\kappa e^{\kappa}}{\sinh(\kappa)}$.  Using
\eqref{eq:almost-geod} together with the fact that the entropy
decreases along the heat flow, we infer that
\begin{equation}\begin{aligned}
\label{eq:approx}
&  \frac{m(\kappa h)}{2} \cW^2(P_h\rho, \nu)
- \frac12 \cW^2(\rho, \nu) - \eps   
\\&\qquad\qquad+  E_\kappa(h) \cH(P_h \rho) - h \cH(\nu) 
   \leq  2 \kappa 
   \int_0^1 \int_0^h  t e^{2\kappa st} \cH(\rho_t^s)
   \dd t \dd s\;,
\end{aligned}\end{equation}
where
$
 E_{\kappa}(h) :=  \int_0^h  e^{2\kappa t} \, dt.
$
Since $\cH$ is bounded, it follows that  
\begin{align*}
 \lim_{h \downarrow 0}
  \frac1h  \int_0^1 \int_0^h  t e^{2\kappa st} \cH(\rho_t^s)
   \dd t \dd s = 0\;.
\end{align*}
Furthermore, 
\begin{align*}
 \lim_{h \downarrow 0} \frac1h \Big( E_\kappa(h) \cH(P_h \rho) - h \cH(\nu) \Big) 
  = \cH(\rho) - \cH(\nu)\;.
\end{align*}
Since $\eps > 0$ is arbitrary, \eqref{eq:approx} implies that
\begin{align*}
  \ddhr\bigg|_{h = 0}  
  \left( \frac{m(\kappa h)}{2} \cW^2(P_h\rho, \nu)\right) 
    +  \cH(\rho) - \cH(\nu) 
    \leq 0\;.
\end{align*}
Taking into account that
\begin{align*}
  \ddhr\bigg|_{h = 0}  
  \left( \frac{m(\kappa h)}{2} \cW^2(P_h\rho, \nu)\right) 
   = \frac{\kappa}{2} \cW^2(\rho, \nu) + \frac12
   \ddhr  \bigg|_{h = 0} \cW^2(P_h\rho, \nu)\;,
\end{align*}
we obtain 
\eqref{eq:EVI} for $t = 0$, which clearly implies \eqref{eq:EVI} for  all $t \geq 0$. 
\end{proof}

The following result, which is used in the proof of Theorem
\ref{thm:Ric-equiv}, is a discrete analogue of \eqref{eq:DS-ineq} and
the proof proceeds along the lines of \cite[Lemma 4.3]{DS08}. Since
the details are slightly different in the discrete setting, we present
a proof for the convenience of the reader.

\begin{lemma}\label{lem:dan-sav}
  Let $\{\rho^s\}_{s \in [0,1]}$ be a smooth curve in $\PX$. For
  each $t \geq 0,$ set $\rho_t^s := e^{st\Delta} \rho^s,$ and let
  $\{\psi_t^s\}_{s \in [0,1]}$ be a smooth curve in $\R^\cX$ satisfying the
  continuity equation
\begin{align*}
  \partial_s \rho_t^s + \nabla \cdot ( \hrho_t^s \cdot \nabla \psi_t^s) =
  0\;, \qquad s \in [0,1]\;.
\end{align*}
Then the identity
\begin{align*}
 \frac12 \partial_t \cA(\rho_t^s, \psi_t^s)
 	 + \partial_s \cH(\rho_t^s) = - s \cB(\rho_t^s, \psi_t^s)
\end{align*}
holds for every $s \in [0,1]$ and $t \geq 0$.
\end{lemma}

\begin{proof}
First of all, we have
\begin{equation}\begin{aligned}\label{eq:deriv-ent}
\partial_s \cH(\rho_t^s)
  & = \bip{1 + \log \rho_t^s \ ,\,  \partial_s \rho_t^s}_{\pi}
\\& = - \bip{1 + \log \rho_t^s \ ,\,
		 \nabla \cdot ( \hrho_t^s \cdot \nabla \psi)}_{\pi}
\\& =  \bip{ \nabla \log \rho_t^s \ ,\,
   				\hrho_t^s \cdot \nabla \psi_t^s}_\pi
\\& =  \bip{ \nabla \rho_t^s\ ,\, \nabla \psi_t^s}_{\pi}
\\& =  - \bip{\psi_t^s \ ,\,\Delta \rho_t^s}_\pi \;.
\end{aligned}\end{equation}
Furthermore, 
\begin{align*}
 \frac12 \partial_t \cA(\rho_t^s, \psi_t^s)
&  =  \bip{ \hrho_t^s \cdot \partial_t\nabla \psi_t^s\ ,\, 
    			\nabla \psi_t^s}_\pi
 + \frac12
			 \bip{ \partial_t\hrho_t^s \cdot \nabla \psi_t^s\ ,\,
    			\nabla \psi_t^s }_\pi
 \\& =: I_1 + I_2\;.			
\end{align*}
In order to simplify $I_1$ we claim that 
\begin{align} 
 \label{eq:partial-ts}
  -\nabla \cdot \big( (\partial_t \hrho_t^s) \cdot \nabla \psi_t^s \big) 
  - \nabla \cdot \big( \hrho_t^s \cdot \partial_t \nabla \psi_t^s \big)
   & =  \Delta \rho_t^s
    - s \Delta\big(\nabla \cdot (\hrho_t^s \cdot\nabla \psi_t^s)\big)\;,\\
  \label{eq:partial-t}
  \partial_t \hrho_t^s & = s \hDelta \rho_t^s\;.
\end{align}
To show \eqref{eq:partial-ts}, note that the left-hand side equals
$\partial_t \partial_s \rho_t^s$, while the right-hand side equals
$\partial_s \partial_t \rho_t^s$. The identity \eqref{eq:partial-t}
follows from a straightforward calculation.

Integrating by parts repeatedly and using \eqref{eq:deriv-ent},
\eqref{eq:partial-ts} and \eqref{eq:partial-t}, we obtain
\begin{align*}
  I_1 
  & = - \bip{ \psi_t^s\ ,\,
    		\nabla \cdot (\hrho_t^s \cdot \partial_t\nabla \psi_t^s)}_\pi
\\  & =   \bip{ \psi_t^s\ ,\, \Delta \rho_t^s}_{\pi}
	    - s \bip{ \psi_t^s\ ,\,
	 \Delta\big(\nabla \cdot (\hrho_t^s \nabla \psi_t^s)\big)}_\pi
  +  \bip{ \psi_t^s\ ,\,
  		   \nabla \cdot
		    \big(  (\partial_t \hrho_t^s) \nabla \psi_t^s \big) }_\pi
\\  & =  - \partial_s \cH(\rho_t^s)
		 +  s  \bip{ \hrho_t^s \cdot \nabla \psi_t^s\ ,\,
   			\nabla \Delta \psi_t^s}_\pi  
	  - s  \bip{\hDelta \rho_t^s \cdot \nabla\psi_t^s\ ,\,
	    \nabla \psi_t^s}_{\pi}\;.
\end{align*}
Taking into account that
\begin{align*}
 I_2 = \frac{s}2 
				\bip{\hDelta \rho_t^s \cdot \nabla \psi_t^s\ ,\,
    					\nabla \psi_t^s }_\pi\;,
\end{align*}
the result follows by summing the expressions for $I_1$ and $I_2$.
\end{proof}

The evolution variational inequality \eqref{eq:EVI} has been extensively studied in the theory of gradient flows in metric spaces \cite{AGS08}. It readily implies a number of interesting properties for the associated gradient flow (see, e.g.,  \cite[Section 3]{DS08}). Among them we single out the following $\kappa$-contractivity property.

\begin{proposition}[$\kappa$-Contractivity of the heat flow]\label{prop:contractility}
Let $(\cX,K)$ be an irreducible and reversible Markov kernel satisfying $\Ric(K) \geq \kappa$ for some $\kappa \in \R$. Then the associated continuous time Markov semigroup $(P_t)_{t \geq 0}$ satisfies 
\begin{align*}
 \cW(P_t \rho, P_t \sigma) \leq e^{-\kappa t} 
  \cW(\rho, \sigma)
\end{align*}
for all $\rho, \sigma \in \PX$ and $t \geq 0$.
\end{proposition}

\begin{proof}
This follows by applying \cite[Proposition 3.1]{DS08} to the functional $\cH$  on the metric space $(\PX,\cW)$.
\end{proof}

\section{Examples}
\label{sec:examples} 

In this section we give explicit lower bounds on the \nl Ricci curvature
in several examples. Moreover, we present a simple criterion (see
Proposition \ref{prop:simplecriterion}) for proving \nl Ricci curvature
bounds. Although the assumptions seem restrictive, the criterion allows to obtain
the sharp Ricci bound for the discrete hypercube. Moreover, it can be
combined with the tensorisation result from Section
\ref{sec:tensorisation} in order to prove Ricci bounds in other
nontrivial situations.  To get started let us consider a particularly
simple example.

\begin{example}[The complete graph]
  Let $\cK^n$ denote the complete graph on $n$ vertices and let $K_n$
  be the simple random walk on $\cK^n$ given by the transition kernel
  $K(x,y)=\frac{1}{n}$ for all $x,y\in\cK^n$. Note that in this case
  $\pi$ is the uniform measure. We will show that
  $\Ric(K_n)\geq\frac12+\frac{1}{2n}$. In view of Theorem
  \ref{thm:Ric-equiv} we have to show
  $\cB(\rho,\psi)\geq(\frac12+\frac{1}{2n})\cA(\rho,\psi)$ for all
  $\rho\in\PXs$ and $\psi\in\R^{\cX}$. Recall the definition
  \eqref{eq:B} of the quantity $\cB$. We calculate explicitly :
  \begin{align*}
    \bip{\hrho  \cdot\,\nabla \psi\ ,\, \nabla \Delta\psi}_\pi~&=~\frac12\frac{1}{n^3}\sum\limits_{x,y,z\in\cX}\hat\rho(x,y)\nabla\psi(y,x)\bigg[\nabla\psi(x,z)-\nabla\psi(y,z)\bigg]\\
    &=~-\frac12\frac{1}{n^2}\sum\limits_{x,y}\hat\rho(x,y)\big(\nabla\psi(x,y)\big)^2~=~-\cA(\rho,\psi)\
    .
  \end{align*}
  With the notation
  $\hat\rho_i(x,y)=\partial_i\theta(\rho(x),\rho(y))$ and using
  equation \eqref{eq:tricktheta1} we obtain further
  \begin{align*}
    \bip{\hDelta\rho\, \cdot\, \nabla \psi, \nabla \psi }_\pi~&=~\frac12\frac{1}{n^3}\sum\limits_{x,y,z}\big(\nabla\psi(x,y)\big)^2\bigg[\hat\rho_1(x,y)(\rho(z)-\rho(x))\\
    &\qquad\qquad\qquad\qquad\qquad\qquad+\hat\rho_2(x,y)(\rho(z)-\rho(y))\bigg]\\
    &=~-\cA(\rho,\psi) + \frac12\frac{1}{n^3}\sum\limits_{x,y,z}\big(\nabla\psi(x,y)\big)^2\bigg[\hat\rho_1(x,y)\rho(z)\\
    &\qquad\qquad\qquad\qquad\qquad\qquad\qquad\qquad+\hat\rho_2(x,y)\rho(z)\bigg]\
    .
  \end{align*}
  Keeping only the terms with $z=x$ (resp. $z=y$) in the last sum and
  using \eqref{eq:tricktheta1} again we see
  \begin{align*}
    \bip{\hDelta\rho\, \cdot\, \nabla \psi, \nabla \psi
    }_\pi~\geq~\left(\frac{1}{n}-1\right)\cA(\rho,\psi)\ .
  \end{align*}
  Summing up we obtain $\cB\geq(\frac12(\frac{1}{n}-1)+1)\cA$, which
  yields the claim.
\end{example}
\medskip

For the rest of this section we let $K$ be an irreducible and reversible
Markov kernel on a finite set $\cX$.  In order to state the criterion
and to perform calculations, it will be convenient to write a Markov
chain in terms of allowed moves rather than jumps from point to point.
 
Let $G$ be a set of maps from $\cX$ to itself (the allowed moves) and
consider a function $c:\cX\times G\to\R_+$ (representing the jump
rates). 

\begin{definition}\label{def:mappingrepresentation}
  We call the pair $(G,c)$ a \emph{mapping representation} of $K$ if
  the following properties hold:
  \begin{enumerate}
  \item The generator $\Delta = K-Id$ can be written in the form
    \begin{align}\label{eq:generator-mappingform}
      \Delta\psi(x)~=~\sum\limits_{\delta\in G} \nabla_\delta \psi(x)
      c(x,\delta)\ ,
    \end{align}
    where
    \begin{align*}
      \nabla_\delta \psi(x) = \psi(\delta x)-\psi(x)\;.
    \end{align*}
  \item For every $\delta\in G$ there exists a unique $\delta^{-1}\in
    G$ satisfying $\delta^{-1}(\delta(x))=x$ for all $x$ with
    $c(x,\delta)>0$.
  \item For every $F:\cX\times G\to\R$ we have
    \begin{align}\label{eq:rev}
      \sum\limits_{x\in\cX,\delta\in
        G}F(x,\delta)c(x,\delta)\pi(x)~=~\sum\limits_{x\in\cX,\delta\in
        G}F(\delta x,\delta^{-1})c(x,\delta)\pi(x)\ .
    \end{align}
  \end{enumerate}
\end{definition}

\begin{remark}\label{rem:CDPP}
  This definition is close in spirit to the recent work \cite{CDPP09},
  where $\Gamma_2$-type calculations have been performed in order to
  prove strict convexity of the entropy along the heat flow in a
  discrete setting. Here, we essentially compute the second
  derivatives of the entropy along $\cW$-geodesics. Since the geodesic
  equations are more complicated than the heat equation, the
  expressions that we need to work with are somewhat more involved.
\end{remark}

Every irreducible, reversible Markov chain has a mapping
representation. In fact, an explicit mapping representation can be
obtained as follows. For $x,y\in\cX$ consider the bijection
$t_{\{x,y\}}:\cX\to\cX$ that interchanges $x$ and $y$ and keeps all
other points fixed. Then let $G$ be the set of all these
``transpositions'' and set $c(x,t_{\{x,y\}})=K(x,y)$ and
$c(x,t_{\{y,z\}})=0$ for $x\notin \{y,z\}$. Then $(G,c)$ defines a
mapping representation. However, in examples it is often more natural
to work with a different mapping representation involving a smaller
set $G$, as we shall see below.

It will be useful to formulate the expressions for $\cA$ and $\cB$ in
this formalism. 
For this purpose, we note that \eqref{eq:generator-mappingform} implies that \begin{align*}
  \sum\limits_{y \in \cX}F(x,y)K(x,y)~=~\sum\limits_{\delta\in G}F(x,\delta x)c(x,\delta)
\end{align*}
for any $F:\cX\times\cX\to\R$ vanishing on the diagonal.
As a consequence we obtain
\begin{align}\label{eq:A-mappingform}
  \cA(\rho,\psi)~=~\frac12 \sum\limits_{x\in\cX,\delta\in G}\big(\nabla_\delta\psi(x)\big)^2\hat\rho(x,\delta x)c(x,\delta)\pi(x)
\end{align}
and 
\begin{equation}\begin{aligned}
\label{eq:T1-mappingform}
  \ip{\hrho\nabla\psi,\nabla \Delta\psi}_\pi~&=~\frac12\sum\limits_{x\in\cX}\sum\limits_{\delta,\eta\in G}\nabla_\delta\psi(x)\bigg[\nabla_\eta\psi(\delta x)c(\delta x,\eta)\\
&\qquad\qquad\qquad -\nabla_\eta\psi(x)c(x,\eta)\bigg]\hat{\rho}(x,\delta x) c(x,\delta)\pi(x)\ .
\end{aligned}\end{equation}
Setting for convenience
$\partial_i\theta(\rho(x),\rho(y))=:\hat\rho_i(x,y)$ for $i=1,2$ we
further get
\begin{equation}\begin{aligned}
\label{eq:T2-mappingform}
  \frac12\ip{\hat\Delta\rho \nabla\psi,\nabla\psi}_\pi
  ~&=~\frac{1}{4}\sum\limits_{x,\delta,\eta}\big(\nabla_\delta\psi(x)\big)^2\bigg[\hat{\rho}_1(x,\delta
  x)\nabla_\eta\rho(x)c(x,\eta)\\
  &\qquad\qquad+\hat{\rho}_2(x,\delta x)\nabla_\eta\rho(\delta x)c(\delta
  x,\eta)\bigg]c(x,\delta)\pi(x)\ .
\end{aligned}\end{equation}
Now the expression for $\cB(\rho,\psi)$ is obtained as the difference
of the preceding two expressions.

We are now ready to state the announced criterion, which shall be used in Examples \ref{ex:circle} and \ref{ex:hypercube} below. Intuitively, condition ii) expresses a certain `spatial homogeneity', saying that the jump rate in a given direction is the same before and after another jump. 

\begin{proposition}\label{prop:simplecriterion}
  Let $K$ be an irreducible and reversible Markov kernel on a finite
  set $\cX$ and let $(G,c)$ be a mapping representation. Consider the
  following conditions:
  \begin{itemize}
  \item[\emph{i)}] $\delta\circ\eta~=~\eta\circ\delta$~, for all $\delta,\eta\in G$,
  \item[\emph{ii)}] $c(\delta x,\eta)~=~c(x,\eta)$~, for all $x\in\cX,\ \delta,\eta\in G$,
  \item[\emph{iii)}] $\delta\circ\delta=id$~, for all $\delta\in G$.
  \end{itemize}
  If i) and ii) are satisfied, then $\Ric(K)\geq 0$. If moreover
  iii) is satisfied, then $\Ric(K)\geq 2C$, where
  \begin{align*}
    C:=\min\{c(x,\delta): x\in\cX,\delta\in G \text{ such that
    }c(x,\delta)>0 \}\ .
  \end{align*}

\end{proposition}

\begin{remark}\label{rem:}
Note that requiring \emph{i)} and \emph{iii)} simultaneously imposes a very strong restriction on the graph associated with $K$. We prefer to state the result in this form in order to give a unified proof which applies both to the discrete circle and the discrete hypercube, with optimal constant in the latter case.
\end{remark}

\begin{proof}[Proof of Proposition \ref{prop:simplecriterion}]
  In view of Theorem \ref{thm:Ric-equiv} it suffices to show that
  $\cB(\rho,\psi)\geq 0$ resp. $\cB(\rho,\psi)\geq 2C\cA(\rho,\psi)$
  for all $\rho\in\PXs$ and $\psi\in\R^{\cX}$. First recall that
  \begin{align*}
  \cB(\rho,\psi)~&=~-\ip{\hrho\nabla\psi,\nabla \Delta\psi}_\pi+\frac12\ip{\hat\Delta\rho \nabla\psi,\nabla\psi}_\pi~=:~ T_1 +T_2\ .
\end{align*}
Using \eqref{eq:T1-mappingform} and conditions i) and ii) we can write
the first summand as
\begin{align*}
  T_1~&=~-\frac12\sum\limits_{x,\delta,\eta}\nabla_{\delta}\psi(x)\bigg[\nabla_{\eta}\psi(\delta x)-\nabla_{\eta}\psi(x)\bigg]\hat{\rho}(x,\delta x) c(x,\delta)c(x,\eta)\pi(x)\\
     &=~-\frac12\sum\limits_{x,\delta,\eta}\nabla_{\delta}\psi(x)\bigg[\nabla_{\delta}\psi(\eta x)
          -\nabla_{\delta}\psi(x)\bigg]\hat{\rho}(x,\delta x)c(x,\delta)c(x,\eta)\pi(x)\ .
\end{align*}
In a similar way we shall write the second summand. Starting from
\eqref{eq:T2-mappingform} and invoking ii) and equation
\eqref{eq:tricktheta1} from Lemma \ref{lem:thetatricks}, we
obtain
\begin{align*}
T_2~&=~\frac14\sum\limits_{x,\delta,\eta}\big(\nabla_{\delta}\psi(x)\big)^2\bigg[\hat{\rho}_1(x,\delta x)\nabla_{\eta}\rho(x)\\
&\qquad\qquad\qquad\qquad\qquad+\hat{\rho}_2(x,\delta x)\nabla_{\eta}\rho(\delta x)\bigg]c(x,\delta)c(x,\eta)\pi(x)\\
   &=~\frac14\sum\limits_{x,\delta,\eta}\big(\nabla_{\delta}\psi(x)\big)^2\bigg[\hat{\rho}_1(x,\delta x)\rho({\eta}x)
 \\&\qquad\qquad\qquad\qquad\qquad
 +\hat{\rho}_2(x,\delta x)\rho(\eta\delta x) - \hat{\rho}(x,\delta x)\bigg]c(x,\delta)c(x,\eta)\pi(x)\ .
 \end{align*}
 Using the reversibility of $K$ in the form of \eqref{eq:rev}, and
 again condition ii) we can write
\begin{align*}
T_2~&=~\frac14\sum\limits_{x,\delta,\eta}\bigg(\big(\nabla_{\delta}\psi(\eta x)\big)^2\bigg[\hat{\rho}_1(\eta x,\delta\eta x)\rho(x) + \hat{\rho}_2(\eta x,\delta\eta x)\rho(\delta x)\bigg] \\
&~\qquad\qquad -\big(\nabla_{\delta}\psi(x)\big)^2\hat\rho(x,\delta x)\bigg) c(x,\delta)c(x,\eta)\pi(x)\ .
\end{align*}
Adding a zero we obtain
\begin{align*}
  T_2~&=~\frac14\sum\limits_{x,\delta,\eta}\left(\big(\nabla_{\delta}\psi(\eta x)\big)^2 -\big(\nabla_{\delta}\psi(x)\big)^2\right)\hat{\rho}(x,\delta x)c(x,\delta)c(x,\eta)\pi(x)\\
  &\quad + \frac14\sum\limits_{x,\delta,\eta}\big(\nabla_{\delta}\psi(\eta x)\big)^2\bigg[\hat{\rho}_1(\eta x,\delta\eta x)\rho(x)+\hat{\rho}_2(\eta x,\delta\eta x)\rho(\delta x)\\
  & \qquad\qquad\qquad\qquad\qquad-\hat{\rho}(x,\delta x)\bigg]c(x,\delta)c(x,\eta)\pi(x)\\
  &=:~T_3 + T_4\ .
\end{align*}
Invoking the inequality \eqref{eq:4term} from Lemma
\ref{lem:thetatricks}, we immediately see that $T_4\geq
0$. Hence we get
\begin{align*}
 \cB(\rho,\psi)~&\geq~T_1 + T_3\\
&=~\frac14\sum\limits_{x,\delta,\eta}\big(\nabla_{\delta}\psi(\eta x)-\nabla_{\delta}\psi(x)\big)^2\hat{\rho}(x,\delta x) c(x,\delta)c(x,\eta)\pi(x)\\
&\geq 0\ .
\end{align*}

If moreover condition iii) is satisfied, the latter estimate can be
improved by keeping only the terms with $\eta=\delta$ in the last
sum. We thus obtain
\begin{align*}
  \cB(\rho,\psi)~&\geq~\frac{C}{4}\sum\limits_{x,\delta}\big(2\nabla_\delta\psi(x)\big)^2\hat{\rho}(x,\delta x)c(x,\delta)\pi(x)\\
  &=~2C\cA(\rho,\psi)\ .
\end{align*}
\end{proof}
Let us now consider some examples to which Proposition
\ref{prop:simplecriterion} can be applied.

\begin{example}[The discrete circle]
\label{ex:circle}
Consider the simple random walk on the discrete circle $C_n=\Z/n\Z$ of $n$ sites
given by the transition kernel $K(m,m-1)=K(m,m+1)=\frac12$ for $m\in
C_n$. We have the following mapping representation for $K$. Set
$G=\{+,-\}$ where $+(m)=m+1$ and $-(m)=m-1$ and let
$c(m,+)=c(m,-)=\frac12$ for all $m$. Proposition
\ref{prop:simplecriterion}  immediately yields that $\Ric(K)\geq
0$.
\end{example}

\begin{example}[The discrete hypercube]
\label{ex:hypercube}
Let $\cQ^n=\{0,1\}^n$ be the hypercube endowed with the usual graph
structure and let $K_n$ be the kernel of the simple random walk on $\cQ^n$. The natural mapping representation is given by $G=\{\delta_1,\dots,\delta_n\}$, where $\delta_i:\cQ^n\to\cQ^n$ is the map that flips the $i$-th
coordinate, and $c(x,\delta_i)=\frac{1}{n}$ for all $x\in\cQ^n$. Here
the criterion from Proposition \ref{prop:simplecriterion} yields
$\Ric(K_n)\geq\frac{2}{n}$. We shall see in Section
\ref{sec:inequalities} that this bound is optimal.

Alternatively we can use the fact that $\cQ^n$ is a product space and
use the tensorisation property Theorem \ref{thm:tensorisation}
below. This will allow to consider asymmetric random walks on the
hypercube as well.
\end{example}

\section{Basic constructions}
\label{sec:tensorisation}

In this section we show how \nl Ricci curvature bounds transform under
some basic operations on a Markov kernel. The main result is Theorem
\ref{thm:tensorisation}, which yields Ricci bounds for product
chains. We start with a simple result, that shows how Ricci bounds
behave under adding laziness.

Let $K$ be an irreducible and reversible Markov kernel on a finite set
$\cX$.  For $\lambda \in (0,1)$ we consider the \emph{lazy} Markov kernel
defined by $K_\lambda := (1- \lambda) I + \lambda K$. Clearly,
$K_\lambda$ is irreducible and reversible with the same invariant
measure $\pi$. With this notation, we have the following result:

\begin{proposition}[Laziness]\label{prop:laziness}
  Let $\lambda \in (0,1)$. If $\Ric(K) \geq\kappa$ for some $\kappa
  \in \R$, then the lazy kernel $K_\lambda$ satisfies
  \begin{align*}
    \Ric(K_\lambda) \geq \lambda \kappa\;.
  \end{align*}
\end{proposition}

\begin{proof}
  Writing $\cA_\lambda$ and $\cB_\lambda$ to denote the lazy versions
  of $\cA$ and $\cB$, a direct calculation shows that
\begin{align*}
 \cA_\lambda(\rho, \psi) =  \lambda \cA(\rho, \psi)\;, \qquad
 \cB_\lambda(\rho, \psi) =  \lambda^2 \cB(\rho, \psi)\,
\end{align*}
for all $\rho \in \PXs$ and $\psi \in \R^{\cX}$. As a consequence, 
\begin{align*}
 \cB_\lambda(\rho, \psi) - \lambda \kappa \cA_\lambda(\rho, \psi)
   =  \lambda^2 (  \cB(\rho, \psi) - \kappa \cA(\rho, \psi) )\;.
\end{align*}
The result thus follows from Theorem \ref{thm:Ric-equiv}.
\end{proof}

We now give a tensorisation property of lower Ricci bounds with
respect to products of Markov chains. For $i=1,\dots,n$, let
$(\cX_i,K_i)$ be an irreducible, reversible finite Markov chain with
steady state $\pi_i$, and let $\alpha_i$ be a non-negative number
satisfying $\sum_{i=1}^n\alpha_i=1$. The product chain $K_\alpha$ on
the product space $\cX=\prod_i\cX_i$ is defined for $\xx = (x_1,
\ldots, x_n)$ and $\yy = (y_1, \ldots, y_n)$ by
\begin{align*}
  K_\alpha(\xx,\yy)~=~
  \begin{cases}
    \sum\limits_{i=1}^n\alpha_iK_i(x_i,x_i)\ , & \mbox{if } x_i = y_i\ \forall i  \ ,\\
    \alpha_iK_i(x_i,y_i)\ , & \mbox{if } x_i\neq y_i \mbox{ and } x_j=y_j\ \forall j\neq i\ ,\\
    0\ , & \mbox{otherwise .}
  \end{cases}
\end{align*}
Note that the steady state of $K_\alpha$ is the product
$\pi=\pi_1\otimes\dots\otimes\pi_n$ of the steady states of $K_i$.

\begin{theorem}[Tensorisation]\label{thm:tensorisation} 
  Assume that $\Ric(K_i)\geq\kappa_i$ for $i=1,\dots,n$. Then we have
  \begin{align*}
    \Ric(K_\alpha)~\geq~\min\limits_i\alpha_i\kappa_i\ .
  \end{align*}
\end{theorem}

\begin{proof}
  In view of Theorem \ref{thm:Ric-equiv} we have to show that for any
  $\rho\in\PXs$ and $\psi:\cX\to\R$ :
  \begin{align*}
    \cB(\rho,\psi)~\geq~(\min\limits_i\alpha_i\kappa_i)\cA(\rho,\psi)\ .
  \end{align*}

  We will use a mapping representation for the Markov kernel
  $K_\alpha$ as introduced in Section \ref{sec:examples}. Let
  $(G_i,c_i)$ be mapping representations of $K_i$ for
  $i=1,\dots,n$.  To each $\delta\in G_i$ we associate a map
  $\bar\delta:\cX\to\cX$ by letting $\delta$ act on the $i$-th
  coordinate. Let us set $G=\bigcup_i\{\bar\delta:\delta\in G_i\}$ and
  define $c:\cX\times G\to\R_+$ by
  \begin{align*}
    c(x,\bar\delta)~:=~\alpha_ic_i(x_i,\delta)\ ,\quad \mbox{for } \delta\in G_i\ .
  \end{align*}
  One easily checks that $(G,c)$ is a mapping representation of
  $K_\alpha$. Recalling the expressions
  \eqref{eq:T1-mappingform},\eqref{eq:T2-mappingform} which constitute
  $\cB$ in mapping representation we write
  \begin{align*}
    \cB(\rho,\psi)~=:~\sum\limits_{x\in\cX,\delta,\eta\in
      G}F(x,\delta,\eta)\ .
  \end{align*}
  Taking into account the product structure of the chain we can write
  \begin{align*}
    \cB(\rho,\psi)~=~\sum\limits_{i,j=1}^n\cB_{i,j}\quad \mbox{ with
    }\quad \cB_{i,j}~=~\sum\limits_{x\in\cX}\sum\limits_{\delta\in
      G_i,\eta\in G_j}F(x,\bar\delta,\bar\eta)\ .
  \end{align*}
  The proof will be finished if we prove the following two assertions:
  \begin{itemize}
  \item[i)] $\cB_{i,j}\geq 0$ for all $i\neq j$ ,
  \item[ii)] $\sum\limits_{i=1}^n\cB_{i,i}\geq(\min\limits_i
    \alpha_i\kappa_i) \cA(\rho,\psi)$ .
  \end{itemize}
  To show i), first note that for $\delta\in G_i$ and $\eta\in G_j$
  the maps $\bar\delta$ and $\bar\eta$ act on different coordinates if
  $i\neq j$. Thus we have
  $\bar\delta\circ\bar\eta=\bar\eta\circ\bar\delta$ and furthermore
  $c(\bar\delta x,\bar\eta)=c(x,\bar\eta)$. Note that these are
  precisely the properties used in the proof Proposition
  \ref{prop:simplecriterion}, hence the assertion here follows from
  the same arguments.

Let us now show ii). We set $\check\cX_i=\prod_{j\neq i}\cX_j$. For
$\check x_i \in \check\cX_i$ we let $\rho^{\check x_i},\psi^{\check
  x_i}:\cX_i\to\R$ denote the functions $\rho$ and $\psi$ where all
variables except $x_i$ are fixed to $\check x_i$. Note that
$\rho^{\check x_i}$ does not necessarily belong to $\cP(\cX_i)$, but
this will be irrelevant in the calculation below, and we shall use
expressions as $\cA(\rho^{\check x_i},\psi^{\check x_i})$ by abuse of
notation. We also set $\check\pi_i=\bigotimes_{j\neq i}\pi_j$. Using
once more the product structure of the chain $c$ we see :
\begin{align*}
  &\cA(\rho,\psi)\\
~&=~\frac12 \sum\limits_{i=1}^n\sum\limits_{x\in\cX,\delta\in G_i}\big(\nabla_{\bar\delta}\psi(x)\big)^2\hat\rho(x,\bar\delta x)c(x,\bar\delta)\pi(x)\\
                 &=~\frac{1}{2}\sum\limits_{i=1}^n\sum\limits_{\check x_i\in\check\cX_i}\sum\limits_{x_i\in\cX_i,\delta\in G_i}\big(\nabla_{\delta}\psi^{\check x_i}(x_i)\big)^2\widehat{\rho^{\check x_i}}(x_i,\delta x_i)\alpha_ic_i(x_i,\delta)\pi_i(x_i)\check\pi_i(\check x_i)\\
                 &=~\sum\limits_{i=1}^n\alpha_i\sum\limits_{\check x_i\in\check\cX_i}\cA_i(\rho^{\check x_i},\psi^{\check x_i})\check\pi_i(\check x_i)\ ,
\end{align*}
where $\cA_i$ (resp. $\cB_i$) denotes the function $\cA$ (resp. $\cB$)
associated with the $i$th chain.  Similarly we obtain
\begin{align*}
  \cB_{i,i}~&=~\alpha_i^2\sum\limits_{\check x_i\in\check\cX_i} \cB_i(\rho^{\check x_i},\psi^{\check x_i})\check\pi_i(\check x_i)\\
         &\geq~\alpha_i^2\kappa_i\sum\limits_{\check x_i\in\check\cX_i} \cA_i(\rho^{\check x_i},\psi^{\check x_i})\check\pi_i(\check x_i)\ ,
\end{align*}
where the last inequality holds by assumption on the curvature bound
for $K_i$. Summing over $i=1,\dots,n$ we obtain ii).
\end{proof}

We shall now apply Theorem \ref{thm:tensorisation} to asymmetric
random walks on the discrete hypercube. Here we consider the case
where $\theta$ is the logarithmic mean.  For $p,q\in(0,1)$ let
$K_{p,q}$ be the Markov kernel on the two point space $\{0,1\}$
defined by $K(0,1)=p,K(1,0)=q$. The asymmetric random walk is the
$n$-fold product chain on $\cQ^n$ denoted by $K_{p,q,n}$ where
$\alpha_i=\frac{1}{n}$). Note that the steady state of $K_{p,q,n}$ is
the Bernoulli measure
\begin{align*}
 \Big( (1-\lambda) \delta_{\{0\}} +  \lambda \delta_{\{1\}} 
  \Big)^{\otimes n}
\end{align*}
with parameter $\lambda = \frac{p}{p+q}$. We then have the following
bound on the \nl Ricci curvature:

\begin{proposition}\label{prop:hypercube}
For $n \geq 1$ we have 
\begin{align*}
\Ric(K_{p,q,n})~\geq~\frac{1}{n}\left(\frac{p+q}{2} + \sqrt{pq}\right)\;.
\end{align*}
\end{proposition}

\begin{proof}
  The two-point space $\cQ^1 = \{0,1\}$ has been analysed in detail in
  \cite{Ma11}.  In particular, \cite[Proposition 2.12]{Ma11} asserts
  that $\Ric (K_{p,q,1}) \geq \kappa_{p,q,n}$, where
\begin{align*}
 \kappa_{p,q,n}~=~ 
     \frac{p+q}{2} +  \inf_{-1 < \beta < 1} \bigg\{\frac{1}{1 - \beta^2}
      \frac{ q(1+\beta) - p(1-\beta) }
    					{\log q(1+\beta) - \log p(1-\beta)}   \bigg\}\;.
\end{align*}
In order to estimate the right-hand side, we use the
logarithmic-geometric mean inequality to obtain for $\beta \in
(-1,1)$,
\begin{align*}
      \frac{1}{1 - \beta^2}
      \frac{ q(1+\beta) - p(1-\beta) }
{\log q(1+\beta) - \log p(1-\beta)} 
~\geq~  \sqrt{\frac{pq}{1 - \beta^2}}
~\geq~ \sqrt{pq}
\end{align*}
We thus infer that $\Ric (K_{p,q,1}) \geq \frac{p+q}{2} + \sqrt{pq}$.
The general bound then follows immediately from Theorem
\ref{thm:tensorisation}.
\end{proof}

We shall see in Section \ref{sec:inequalities} that this bound is
sharp if $p = q$. If $p \neq q$, it should be possible to improve this
bound by obtaining a sharper bound in the minimisation problem in the
proof above.

\medskip As another application of the tensorisation result, we prove
nonnegativity of the \nl Ricci curvature for the simple random walk on a
discrete torus of arbitrary size in any dimension $d \geq 1$.

Let $\cc := \{ c_n \}_{n=1}^d$ be a sequence of natural numbers and consider the discrete torus
\begin{align*}
  T_\cc := C_{c_1} \times \ldots \times C_{c_d}\;.
\end{align*}
The simple random walk $K_\cc$ on $T_\cc$ is the $d$-fold product of simple random walks on the circles of length $c_1, \ldots, c_d$. 

\begin{proposition}[$d$-dimensional torus]\label{prop:torus}
For any $d \geq 1$ and $\cc := \{ c_n \}_{n=1}^d \in \N^d$ we have 
\begin{align*}
\Ric(K_\cc) \geq 0\;.
\end{align*}
\end{proposition}

\begin{proof}
This follows from Example \ref{ex:circle} and Theorem \ref{thm:tensorisation}.
\end{proof}

\section{Functional inequalities}
\label{sec:inequalities}

The aim of this section is to prove discrete counterparts to the
celebrated theorems by Bakry-\'Emery and Otto-Villani. Along the way
we prove a discrete version of the HWI-inequality, which relates the
$L^2$-Wasserstein distance to the entropy and the Fisher information.
As announced in the introduction, we shall follow the approach from
Otto-Villani, which relies on the fact that the heat flow is the
gradient flow of the entropy. Therefore, the role of the
$L^2$-Wasserstein distance will be taken over by the distance $\cW$.

We fix a finite set $\cX$ and an irreducible and reversible Markov
kernel $K$ with steady state $\pi$. Recall that the relative entropy
of a density $\rho\in\cP(\cX)$ is defined by
\begin{align*}
  \cH(\rho)~=~\sum\limits_{x\in\cX}\rho(x)\log\rho(x)\pi(x)\ .
\end{align*}
As before, we consider a discrete analogue of the Fisher information, given for $\rho\in\PXs$ by
\begin{align*}
  \cI(\rho)~=~\frac{1}{2}\sum\limits_{x,y\in\cX} \big(\rho(x)-\rho(y)\big) \big(\log\rho(x)-\log\rho(y)\big) K(x,y)\pi(x)\ .
\end{align*}
If $\rho(x)=0$ for some $x\in \cX$, we set $\cI(\rho)=+\infty$. Note that
this quantity can be rewritten in the form
$\cI(\rho)=\|\nabla\log\rho\|^2_\rho$ using the definition of the
logarithmic mean. The relevance of $\cI$ in this setting is due to the fact that it describes the entropy dissipation along the heat flow:
\begin{align}\label{eq:H-derivative}
\ddt\cH(P_t\rho)~&=~-\cI(P_t\rho)\;.
\end{align}

The following proposition gives an upper bound for the speed of the heat flow measured in the metric $\cW$.

\begin{proposition}\label{prop:HW-derivative}
  Let $\rho,\sigma\in\PX$. For all $t > 0$ we have
\begin{align}
  \label{eq:W-derivative}
  \ddtr\cW(P_t\rho,\sigma)~&\leq~\sqrt{\cI(P_t\rho)}\;.
\end{align}
In particular,  the metric derivative of the heat flow with respect to $\cW$ satisfies $\abs{(P_t\rho)'} \leq\sqrt{\cI(P_t\rho)}$.
If $\rho$ belongs to $\cP_*(\cX)$, then \eqref{eq:W-derivative} holds at $t=0$ as well.
\end{proposition}

\begin{proof}
  Let us set $\rho_t:=P_t\rho$. Elementary Markov chain theory
  guarantees that $\rho_t\in\cP_*(\cX)$ for all $t>0$ and that the map
  $t\mapsto\rho_t$ is smooth. To prove \eqref{eq:W-derivative} we use the
  triangle inequality and obtain
  \begin{align*}
    \ddtr\cW(\rho_t,\sigma)~&=~\limsup\limits_{s\searrow 0}\frac{1}{s}\left(\cW(\rho_{t+s},\sigma)-\cW(\rho_t,\sigma)\right)\\
    &\leq~\limsup\limits_{s\searrow
      0}\frac{1}{s}\cW(\rho_t,\rho_{t+s})\ .
  \end{align*}
  Note that the couple $(\rho_r,-\log\rho_r)_{r\in[0,1]}$ solves the
  continuity equation \eqref{eq:cont}. From the definition of $\cW$ we
  thus obtain the estimate
  \begin{align*}
    \limsup\limits_{s\searrow 0}\frac{1}{s}\cW(\rho_t,\rho_{t+s})~&\leq~\limsup\limits_{s\searrow 0}\frac{1}{s}\int\limits_t^{t+s}\parallel\nabla\log\rho_r\parallel_{\rho_r}\dd r \nonumber\\
    &=~\limsup\limits_{s\searrow 0}\frac{1}{s}\int\limits_t^{t+s}\sqrt{\cI(\rho_r)}\dd r \nonumber\\
    &=~\sqrt{\cI(\rho_t)}\ .
  \end{align*}
  The last equality holds since $r\mapsto \sqrt{\cI(\rho_r)}$ is a
  continuous function.
\end{proof}

Let us now recall from  Section \ref{sec:intro} the functional inequalities that will be
studied.
Recall that $\one \in \PX$ denotes the density of the stationary
distribution, which is everywhere equal to $1$.

\begin{definition}\label{def:MLSI}
  The Markov kernel $K$ satisfies 
\begin{enumerate}
\item   a \emph{modified logarithmic
    Sobolev} inequality with constant $\lambda>0$ if for all $\rho\in\PX$
  \begin{align}
  \tag{MLSI($\lambda$)}
  \label{eq:MLSI}
    \cH(\rho)~\leq~\frac{1}{2\lambda}\cI(\rho)\ .
  \end{align}

\item an \emph{H$\cW$I} inequality with
  constant $\kappa\in\R$ if for all
  $\rho\in\PX$
  \begin{align}\label{eq:HWI}
  \tag{{H$\cW$I}($\kappa$)}
    \cH(\rho)~\leq~\cW(\rho,\one)\sqrt{\cI(\rho)}-\frac{\kappa}{2}\cW(\rho,\one)^2\ .
  \end{align}

\item a \emph{modified Talagrand}
  inequality with constant $\lambda > 0$ if for all $\rho\in\PX$
  \begin{align}\label{eq:T}
  \tag{{T}$_{\cW}$($\lambda$)}
    \cW(\rho,\one)~\leq~\sqrt{\frac{2}{\lambda}\cH(\rho)}\ .
  \end{align}

\item a \emph{Poincar\'e} inequality with constant $\lambda > 0$ 
  if for all $\phi\in\R^{X}$ with $\sum_x \phi(x)\pi(x)=0$
  \begin{align}\label{eq:P}
  \tag{{P}($\lambda$)}
    \|{\phi}\|^2_\pi~\leq~\frac{1}{\lambda}\|{\nabla \phi}\|^2_\pi\ .
  \end{align}
\end{enumerate}  
\end{definition}

The following result is a discrete analogue of a result by Otto and
Villani \cite{OV00}.

\begin{theorem}\label{thm:Ric2HWI}
  Assume that $\Ric(K)\geq\kappa$ for some $\kappa\in\R$. Then
  $K$ satisfies  $\HcWI(\kappa)$.
\end{theorem}
\begin{proof}
  Fix $\rho\in\PX$. Without restriction we can assume that $\rho>0$
  since otherwise $\cI(\rho)=+\infty$ and there is nothing to
  prove. Let $\rho_t=P_t\rho$ where $P_t=e^{t(K-I)}$ is the heat
  semigroup. From Theorem \ref{thm:Ric-equiv} and the lower bound on
  the Ricci curvature we know that the curve $(\rho_t)$ satisfies
  EVI($\kappa$), i.e., equation \eqref{eq:EVI}. Choosing in particular
  $\nu=\one$ and $t=0$ in the EVI we obtain the inequality
  \begin{equation*}
    \cH(\rho)~\leq~-\frac12 \ddtrz \cW(\rho_t,\one)^2 -\frac{\kappa}{2}\cW(\rho,\one)^2\;.
  \end{equation*}
  To finish the proof we show that
  \begin{align*}
    -\frac12 \ddtrz\cW(\rho_t,\one)^2~\leq~\cW(\rho,\one) \sqrt{\cI(\rho)}\;.
  \end{align*}
  Indeed, using the triangle inequality
  we estimate
  \begin{align*}
    -\frac12 \ddtrz \cW(\rho_t,\one)^2~&=~\liminf\limits_{s\searrow 0}\frac{1}{2s}\left(\cW(\rho,\one)^2-\cW(\rho_{s},\one)^2\right)\\
                                           &\leq~~\limsup\limits_{s\searrow 0}\frac{1}{2s}\left(\cW(\rho,\rho_{s})^2 +2\cW(\rho,\rho_{s})\cdot
                                           \cW(\rho,\one)\right)\ ,
  \end{align*}
  Using the estimate \eqref{eq:W-derivative} from Proposition
  \ref{prop:HW-derivative} with $\sigma=\rho$ and $t=0$ we see that
  the second term on the right hand side is bounded by $\cW(\rho,\one)
  \sqrt{\cI(\rho)}$ while the first term vanishes.
\end{proof}

The following result is now a simple consequence.

\begin{theorem}[Discrete Bakry-\'{E}mery Theorem]\label{thm:Ric2LSI}
  Assume that $\Ric(K)\geq\lambda$ for some $\lambda>0$. Then $K$
  satisfies $\mLSI(\lambda)$.
\end{theorem}
\begin{proof}
  By Theorem \ref{thm:Ric2HWI} $K$ satisfies
  $\HcWI(\lambda)$. From this we derive $\mLSI(\lambda)$ by an
  application of Young's inequality :
  $$x y~\leq~c x^2 +\frac{1}{4c}y^2\qquad \forall x,y\in\R\;, \ c>0\ ,$$
  in which we set $x=\cW(\rho,1),\ y=\sqrt{\cI(\rho)}$ and
  $c=\frac{\lambda}{2}$.
\end{proof}

\begin{theorem}[Discrete Otto-Villani Theorem]\label{thm:LSI2T}
  Assume that $K$ satisfies\\ 
  $\mLSI(\lambda)$ 
  for some
  $\lambda>0$. Then $K$ also satisfies
  $\mTal(\lambda)$.
\end{theorem}
\begin{proof}
  It is sufficient to prove that  $\mTal(\lambda)$ holds for any
  $\rho\in\cP_*(\cX)$. The inequality for general $\rho$ can then be
  obtained by an easy approximation argument taking into account the
  continuity of $\cW$ with respect to the Euclidean metric.
  
  So fix $\rho\in\cP_*(\cX)$ and set $\rho_t=P_t\rho$. First note that as
  $t\to\infty$, we have
  \begin{align}\label{eq:asymptotic}
    \cH(\rho_t)\to 0 \quad\mbox{ and }\quad \cW(\rho,\rho_t)\to
    \cW(\rho,\one)\;.
  \end{align}
  Indeed, by elementary Markov chain theory, we know that as $t \to
  \infty$, one has $\rho_t \to \one$ in, say, the Euclidean
  distance. The claim follows immediately from the continuity of $\cH$
  and $\cW$ with respect to the Euclidean distance, the latter being a
  consequence of, e.g., Proposition \ref{prop:Wasserstein-upperbound}.
  
  We now define the function $F:\R_+\to\R_+$ by
  \begin{align*}
    F(t)~:=~\cW(\rho,\rho_t) + \sqrt{\frac{2}{\lambda}\cH(\rho_t)}\;.
  \end{align*}
  Obviously we have $F(0)=\sqrt{\frac{2}{\lambda}\cH(\rho)}$ and by
  \eqref{eq:asymptotic} we have that $F(t)\to\cW(\rho,\one)$ as
  $t\to\infty$. Hence it is sufficient to show that $F$ is
  non-increasing. To this end we show that its upper right derivative
  is non-positive. If $\rho_t\neq \one$ we deduce from Proposition
  \ref{prop:HW-derivative} that
  \begin{align*}
    \ddtr F(t)~\leq~\sqrt{\cI(\rho_t)} - \frac{\cI(\rho_t)}{\sqrt{2\lambda\cH(\rho_t)}}~\leq~0\ ,
  \end{align*}
  where we used  $\mLSI(\lambda)$ in the last inequality. If $\rho_t=\one$, then
  the relation also holds true, since this implies that $\rho_r=\one$ for all
  $r\geq t$.  
\end{proof}
  
In a classical continuous setting it is well known that a logarithmic
Sobolev inequality implies a Poincar\'e inequality by
linearisation. Let us make this explicit in the present discrete
context. Fix $\phi\in\R^\cX$ satisfying
$\sum_x \phi(x)\pi(x)=0$ and for sufficiently small
$\eps>0$ set $\rho^\eps=\one+\eps \phi\in\PXs$. One easily checks that
as $\eps\to0$ we have:
\begin{align*}
  \frac{1}{\eps^2}\cH(\rho^\eps)\longrightarrow
   \frac12
    \|\phi \|^2_\pi\
  ,\qquad
  \frac{1}{\eps^2}\cI(\rho^\eps)\longrightarrow \|\nabla \phi\|
    ^2_\pi\;.
\end{align*}
Thus assuming \ref{eq:MLSI} holds and applying it to $\rho^\eps$ we
get the Poincar\'e inequality \ref{eq:P}. In \cite{OV00} it has been
shown that the Poincar\'e inequality can also be obtained from
Talagrand's inequality by linearisation. The same is true for the 
modified Talagrand inequality involving the distance
$\cW$.

\begin{proposition}\label{prop:T2P}
  Assume that $K$ satisfies \ref{eq:T} for some $\lambda>0$. Then
  $K$ also satisfies \ref{eq:P}. In particular, $\Ric(K)\geq\lambda$
  implies $\Poinc(\lambda)$.
\end{proposition}

\begin{proof}
  Assume that \ref{eq:T} holds and let us show  $\Poinc(\lambda)$. The
  second assertion of the proposition then follows from Theorem
  \ref{thm:Ric2LSI} and Theorem \ref{thm:LSI2T}. So fix $\phi\in\R^\cX$
  satisfying $\sum_x\phi(x)\pi(x)=0$ and for sufficiently small
  $\eps>0$ set $\rho^\eps=\one+\eps \phi\in\PXs$. Let
  $(\rho_\cdot^\eps,V_\cdot^\eps)\in\CE_{1}'(\rho^\eps,\one)$ be an action minimizing
  curve. Now we write using the continuity equation
  \begin{align*}
    \sum_x\phi(x)^2\pi(x)~&=~\frac{1}{\eps}\left[\sum_x\phi(x)\big(\rho^\eps(x)-1\big)\pi(x)\right]\\
                      &=~\frac{1}{2\eps}\int_0^1\sum_{x,y}\nabla \phi(x,y)V^\eps_t(x,y)K(x,y)\pi(x)\dd t\ .
  \end{align*}
  Using H\"older's inequality we can estimate
  \begin{align*}
    \sum_x\phi(x)^2\pi(x)~&\leq~\frac{1}{\eps}\left(\int_0^1\norm{\nabla \phi}^2_{\rho^\eps_t}\dd t\right)^\frac12\left(\int_0^1\cA'(\rho_t^\eps,V^\eps_t)\dd t\right)^\frac12\\
                      &=~\left(\frac12\sum_{x,y}\big(\nabla \phi(x,y)\big)^2g^\eps(x,y)K(x,y)\pi(x)\right)^\frac12 \frac{1}{\eps}\cW(\rho^\eps,\one)\ ,
  \end{align*}
  where $g^\eps\in\R^{\cX\times\cX}$ is defined by
  $g^\eps(x,y)=\int_0^1\hat\rho^\eps_t(x,y)\dd t$. Using  $\mTal(\lambda)$ we
  arrive at
  \begin{align*}
    \norm{\phi}^2_\pi~&\leq~\norm{(\nabla \phi)\sqrt{g^\eps}}_\pi\frac{1}{\eps}\sqrt{\frac{2}{\lambda}\cH(\rho^\eps)}\ .
  \end{align*}
  The proof will be finished if we show that as $\eps$ goes to $0$
 \begin{align*}
   \frac{1}{\eps}\sqrt{\frac{2}{\lambda}\cH(\rho^\eps)}~\longrightarrow~\sqrt{\frac{1}{\lambda}}\norm{\phi}_\pi\
   ,\qquad \norm{(\nabla
     \phi)\sqrt{g^\eps}}_\pi~\longrightarrow~\norm{\nabla \phi}_\pi\ .
 \end{align*}
 As before the first statement is easily checked. For the second
 statement it is sufficient to show that $\rho_t^\eps\to \one$
 uniformly in $t$ as $\eps\to 0$, as this implies that $g^\eps\to
 1$. Since $\cW(\rho^\eps,\one)\to 0$ as $\eps\to 0$, this follows
 immediately from the estimate
 \begin{align*}
   \cW(\rho^\eps,\one)~\geq~\sup\limits_t\cW(\rho_t^\eps,\one)~\geq~\sup\limits_t\sum\limits_x\pi(x)\abs{\rho^\eps_t(x)-1}\ ,
 \end{align*}
 where we used that $(\rho_t^\eps)_{t\in[0,1]}$ is a geodesic and the
 fact that $\cW$ is an upper bound for the total variation distance
 (see Proposition \ref{prop:Wasserstein-lowerbound}).
\end{proof}

In the following result we use the probabilistic notation
\begin{align*}
 \E_\pi[\phi] = \sum_{x\in \cX} \phi(x) \pi(x)
\end{align*}
for functions $\phi : \cX \to \R$.

\begin{proposition}\label{prop:W1}
  Assume that $K$ satisfies  $\mTal(\lambda)$ for some $\lambda>0$. Then the
  $T_1(2\lambda)$ inequality holds with respect to the graph distance:
  \begin{align}\label{eq:T1}
    W_{1,g}(\rho,\one)~\leq~\sqrt{\frac{1}{\lambda}\cH(\rho)}\ .
  \end{align}
Furthermore, the sub-Gaussian inequality 
\begin{align}\label{eq:sub-Gauss}
  \E_\pi \big[e^{t(\phi- \E_\pi [\phi])}\big] \leq \exp\Big(\frac{t^2}{4\lambda}\Big)
\end{align}
holds for all $t > 0$ and every function $\phi : \cX \to \R$ that is
$1$-Lipschitz with respect to the graph distance on $\cX$.
\end{proposition}

\begin{proof}
  The $T_1$-inequality \eqref{eq:T1} follows immediately from Proposition \ref{prop:Wasserstein-lowerbound}. The
  inequalities \eqref{eq:T1} and \eqref{eq:sub-Gauss} are equivalent,
  as has been shown in \cite{BG99}.
\end{proof}

Arguing again exactly as in \cite{OV00}, we infer that a modified
Talagrand inequality implies a modified log-Sobolev inequality (with
some loss in the constant), provided that the \nl Ricci curvature is not
too bad.

\begin{proposition}\label{prop:reverse-OV}
  Suppose that $K$ satisfies  $\mTal(\lambda)$ for some $\lambda > 0$ and
  that $\Ric(K) \geq \kappa$ for some $\kappa > - \lambda$. Then $K$
  satisfies $\mLSI(\tilde\lambda)$, where
\begin{align*}
 \tilde\lambda = \max\bigg\{ 
     \frac{\lambda}{4}\Big(1 + \frac{\kappa}{\lambda}\Big)^2 \ , \ 
 		\kappa  \bigg\}\;.
\end{align*}
\end{proposition}

\begin{proof}
  This is an immediate consequence of the $\HcWI(\kappa)$-inequality and
  an elementary computation (see \cite[Corollary 3.1]{OV00}).
\end{proof}

As an application of the results proved in this section, we will show how \nl Ricci curvature bounds can be used to recover functional inequalities with sharp constants in an important example.

\begin{example}[Discrete hypercube]
\label{ex:hypercube-again}
In Example \ref{ex:hypercube} and Proposition \ref{prop:hypercube} we
proved that the Markov kernel $K_n$ associated with the simple random walk
on the discrete hypercube $\cQ^n=\{0,1\}^n$ has \nl Ricci curvature
bounded from below by $\frac{2}{n}$. Applying Theorem \ref{thm:Ric2LSI} and
Proposition \ref{prop:W1} in this setting we obtain the following
result. We shall write $y \sim x$ if $K(x,y) > 0$.

\end{example}

\begin{corollary}
The simple random walk on $\cQ^n$ has the following properties:
\begin{enumerate}
\item the modified log-Sobolev inequality
 $\mLSI(\frac{2}{n})$ holds, i.e., for all $\rho\in\cP_*(\cQ^n)$ we have 
$$\sum_{x\in \cQ^n}\rho(x)\log\rho(x)~\leq~\frac{1}{8}\sum\limits_{x \in \cQ^n, y\sim x}\left(\rho(x)-\rho(y)\right) \left(\log\rho(x)-\log\rho(y)\right)\ .$$ 
\item the Poincar\'e inequality $\Poinc(\frac2n)$ holds, i.e., for
  all $\phi:\cQ^n \to \R$ we have
$$\sum_{x\in \cQ^n}\phi(x)^2
~\leq~\frac{1}{4}\sum\limits_{x \in \cQ^n, y\sim x}\left(\phi(x)-\phi(y)\right)^2\ .$$ 
\item The sub-Gaussian inequality \eqref{eq:sub-Gauss} holds with $\lambda = \frac2n$. 
\end{enumerate}
\end{corollary}

In all cases the constants are optimal (see \cite[Example 3.7]{BT06}
and \cite[Proposition 2.3]{BHT06} respectively). Moreover, the optimality in (3) implies that the constant  $\lambda = \frac2n$ in the modified Talagrand inequality for the discrete cube is sharp as well.

We finish the paper by remarking that modified logarithmic Sobolev inequalities for appropriately rescaled product chains on the discrete hypercube $\{-1,1\}^n$ can be used to prove a similar inequality for Poisson measures by passing to the limit $n \to \infty$ (see \cite[Section 5.4]{Le01} for an argument along these lines involving a slightly different modified log Sobolev inequality). All of the functional inequalities in Theorem \ref{thm:main-inequalities} are compatible with this limit. However, the sub-Gaussian estimate will (of course) not hold for the limiting Poisson law. This does not contradict the results in this section, since the sub-Gaussian estimates here are obtained using the lower bound for $\cW$ in terms of $W_1$, which relies on the normalisation assumption $\sum_{y} K(x,y) = 1$, which does not hold in the Poissonian limit.

\bibliographystyle{plain}
\bibliography{ricci}

 \end{document}